\documentclass{article}
\usepackage[utf8]{inputenc}
\usepackage{arxiv}
\usepackage[english]{babel}
\usepackage{graphicx}
\usepackage{listings}
\usepackage{algorithm}
\usepackage{algorithmic}
\usepackage[all]{xy}
\usepackage{tikz-cd}
\usepackage{hyperref}
\usetikzlibrary{cd}
\usepackage{amsmath}
\usepackage{dsfont}
\usepackage{ stmaryrd }
\usepackage{calrsfs}
\usepackage{mathrsfs}
\usepackage{ amssymb }
\usepackage{multirow}
\usepackage{framed}
\usepackage{tabularx, array, caption}

\usepackage{siunitx}
\usepackage{lipsum}
    \usepackage[T1]{fontenc}
    \usepackage{mathtools}
    \usepackage{amsthm}
    \usepackage[framemethod=TikZ]{mdframed}
    
\usepackage{caption}
\usepackage{threeparttable}
\usepackage{booktabs}
\usepackage{siunitx}

\usepackage{fullpage}

\usepackage{authblk}

\DeclareMathAlphabet\mathbfcal{OMS}{cmsy}{b}{n}

\DeclareMathOperator{\LList}{LList}
\DeclareMathOperator{\bool}{bool}

\DeclareMathOperator{\Ell}{Ell}
\DeclareMathOperator{\Jac}{Jac}

\DeclareMathOperator{\im}{im}
\DeclareMathOperator{\st}{st}

\DeclareMathOperator{\Hom}{Hom}

\DeclareMathOperator{\diag}{diag}

\DeclareMathOperator{\Gal}{Gal}
\DeclareMathOperator{\Cl}{Cl}

\DeclareMathOperator{\Aut}{Aut}

  \newcommand{\eq}[1][r]
   {\ar@<-3pt>@{-}[#1]
    \ar@<-1pt>@{}[#1]|<{}="gauche"
    \ar@<+0pt>@{}[#1]|-{}="milieu"
    \ar@<+1pt>@{}[#1]|>{}="droite"
    \ar@/^2pt/@{-}"gauche";"milieu"
    \ar@/_2pt/@{-}"milieu";"droite"}

\newcommand{\ST}{{\mathrm {ST}}}
\newcommand{\Res}{{\mathrm {Res}}}
\newcommand{\lra}{\longrightarrow}
\newcommand{\ra}{{\rightarrow}}

\newcommand{\LListFM}{\LList_{\text{FM}-\QQ}}

\newcommand{\fv}{\mathfrak{v}}

\newcommand{\fa}{\mathfrak{a}}

\newcommand{\fb}{\mathfrak{b}}
\newcommand{\fs}{\mathfrak{s}}

\newcommand{\QQ}{\mathds{Q}}
\newcommand{\Qbar}{\overline{\QQ}}
\newcommand{\GalQK}{\Gal\left(\Qbar/K\right)}
\newcommand{\GalKQ}{\Gal\left(K/\QQ\right)}
\newcommand{\GalQQ}{\Gal\left(\Qbar/\QQ\right)}

\newcommand{\PP}{\mathds{P}}

\newcommand{\trsp}[1]{{}^{t}\!#1}

\newcommand{\ZZ}{\mathds{Z}}
\newcommand{\RR}{\mathds{R}}

\newcommand{\AR}{\mathcal{A}_R}
\newcommand{\ARQ}{\mathcal{A}_{R,\QQ}}

\newcommand{\RLat}{\mathbfcal{L}_R}
\newcommand{\RhLat}{\RLat^{h,int}}
\newcommand{\AVR}{\mathbfcal{A}_R}
\newcommand{\PAV}{\mathbfcal{A}_R^p}
\newcommand{\T}{\mathbf{T}}
\newcommand{\PT}{\T^p}
\newcommand{\PTR}{\mathbfcal{T}_R^p}
\newcommand{\R}{\mathbf{F}}

\newcommand{\Rh}{\R_h}
\newcommand{\NN}{\mathds{N}}
\newcommand{\CC}{\mathds{C}}

\newcommand{\End}{\text{End}}

\newcommand{\an}{_{\text{an}}}
\newcommand{\rat}{_{\text{rat}}}

\newcommand{\set}[1]{\left\{#1\right\}}
\newcommand{\bracket}[1]{\left[#1\right]}
\newcommand{\parent}[1]{\left(#1\right)}
\newcommand{\comm}[1]{}

\newtheorem{prop}{Proposition}
\newtheorem{theo}{Theorem}

\newtheorem{coro}{Corollary}
\newtheorem{rem}{Remark}
\newtheorem{lem}{Lemma}
\def\yes{1}
\def\no{0}
\def\makeitniceforPhDmanuscript{\no}
\if\makeitniceforPhDmanuscript\yes
\fi

\title{Polarized products of elliptic curves with complex multiplication and field of moduli $\QQ$}
\date{March 2022}

\begin{document}

\author{Fabien Narbonne\thanks{Univ Rennes, CNRS, IRMAR - UMR 6625, F-35000
Rennes, %
  France.\\\textit{Email address: }fabien.narbonne@univ-rennes1.fr}\\\textit{with an appendix of }Francesc Fité\textit{ and} Xavier Guitart\thanks{Departament de matem\`atiques i inform\`atica, Universitat de Barcelona}}

\maketitle

\begin{abstract}
    Let $R$ be the maximal order in a quadratic imaginary field $K$. We give an equivalence of categories between the category of polarized abelian varieties isomorphic to a product of elliptic curves over $\CC$ with complex multiplication (CM) by $R$ and the category of integral hermitian $R$-lattices. Then we apply this equivalence to enumerate all the genus $2$ and $3$ curves with field of moduli $\QQ$ and with Jacobian isomorphic to a product of elliptic curves with CM by $R$.
\end{abstract}

\section*{Introduction}

Let $E$ be an elliptic curve over $\CC$ with complex multiplication by a maximal order $R$ in a quadratic imaginary field $K$. Then $E$ admits a model over $\overline{\QQ}$. It even admits one over $\QQ(j(E))$, which has degree $\#\Cl(R)$ over $\QQ$, and not over any sub-extension. It may then be surprising that powers of CM elliptic curves may be defined over smaller fields than expected, sometimes even over $\QQ$. In \cite[Theorem~1.1~and~1.2]{FG20} the authors show  for instance that there are abelian surfaces defined over $\QQ$, $\overline{\QQ}$-isogenous to the square of a CM elliptic curve, for exactly $45$ discriminants even though there are only $13$ CM elliptic curves over $\QQ$. From this they deduce, \cite[Corollary 1.3]{FG20} that there are exactly $92$ $\Qbar$-endomorphism algebras of geometrically split abelian surfaces over $\QQ$. This study is motivated by the conjecture on the possible finiteness of the set of endomorphism rings of abelian varieties of a given dimension over a fixed degree extension field of $\QQ$. 

A weaker requirement is to ask for the field of moduli of a (polarized) abelian variety to be $\QQ$.  In \cite{GHR} the authors give the finite list of indecomposable principally polarized abelian surfaces (also known as Jacobian of genus $2$ curves) with field of moduli $\QQ$ which are isomorphic to $E^2$.

In the present article we address the case of abelian varieties with field of moduli $\QQ$ isomorphic over $\CC$ to products $E_1 \times \cdots \times E_g$, $g > 1$, of elliptic curves with CM by a maximal order $R$ and we give an exhaustive list of these principally polarized abelian varieties for the case $g=2$ in Table \ref{g2} and for $g=3$ in Table \ref{g3}. Before explaining the structure of the paper and our strategy, we mention the following natural generalization to abelian varieties isomorphic to products $E_1 × ··· × E_g, g > 1$, of elliptic curves with CM by possibly distinct orders in the same imaginary quadratic field $K$. By \cite[Theorem 2]{kani}, this is equivalent to considering abelian varieties isogenous to $E^g$ over $\CC$ for a given elliptic curve E with CM by K. In my PhD thesis, I presented  heuristic results for $g = 2$ in this direction, see \cite[Table A.1, A.2 and A.3]{Nar22}. Note that  in \cite[Théorème 2.2.33]{Nar22}, the equivalence of categories of Theorem \ref{equivR} is already proved for arbitrary orders. To finish the proof, one would need to check the validity of Theorem \ref{KQ} and Theorem \ref{QK} in this generality. We also hope to address the finer question the existence of a model over $\QQ$ for principally polarized abelian variety with field of moduli $\QQ$. This theoretical result would also lead to the more refined project of having certified models over $\QQ$ when the descent is possible. Very recent results contained in \cite{LRR} would help to achieve this. Some certified models of genus $2$ curves over $\QQ$ with Jacobian isomorphic to a product of CM elliptic curves can be found in \cite{FFG}.

In Section~\ref{sect1} we recall some properties of the classical equivalence of categories between the complex abelian varieties and polarizable tori and how it behaves with respect to polarizations. Polarizations on abelian varieties give rise to positive definite hermitian forms with a compatibility condition on the lattice.

In Section~\ref{sect2}, we restrict the equivalence of categories to the abelian varieties isomorphic to the product of elliptic curves with complex multiplication by $R$. For these abelian varieties, we associate the complex tori $V/\Gamma$, where $V$ is a $\CC$-vector space and $\Gamma$ is a lattice in $V$. The lattice $\Gamma$ can be endowed with the structure of an $R$-module, providing additional structure than the $\ZZ$-module one it already has. The corresponding hermitian form endows the $R$-module $\Gamma$ with the structure of an \textit{integral hermitian lattice} (up to rescaling the hermitian form by a constant that only depends on $R$). This restriction leads to an equivalence of categories between polarized abelian varieties isomorphic to a product of elliptic curves with CM by $R$ and integral hermitian $R$-lattices (Theorem \ref{equivR}). Moreover, hermitian forms corresponding to principal polarization give \textit{unimodular lattices} through the equivalence. We may note that a similar functor is developed in \cite{JKP+18} in a much wider framework since it is defined for any field, not only $\CC$. Under some conditions, it is also an equivalence of categories. However, we chose to focus our attention on a much simpler functor since we will need to refine it in the next section.

The goal of Section~\ref{sect3} is to translate the action of $\Gal(\overline{\QQ}/\QQ)$ on abelian varieties $\overline{\QQ}$-isomorphic to $E_1 \times \cdots \times E_g$ into the category of integral hermitian lattices through the previous functor. The action of $\Gal(\overline{\QQ}/\QQ)$ can be decomposed into two steps: the action of $\Gal(\overline{\QQ}/K)$ and the action of $\Gal(K/\QQ)$, the complex conjugation. For the first one, we use the existence of a surjective morphism $F\colon \Gal(\overline{\QQ}/K)\rightarrow \Cl(R)$ such that the action of $\sigma\in\Gal(\overline{\QQ}/K)$ on abelian varieties corresponds to the tensor product of the associated lattice by a representative of $F(\sigma)$ (and a rescaling of the hermitian form).

Still, for both actions, we need to rigidify the choice of the target objects under the previous equivalences of categories, which are defined only up to $\CC$-isomorphisms. Indeed, unlike \cite{GHR} where the compatibilities between the various abelian varieties and their conjugate were obvious, we could not find a simple way to impose them from abstract nonsense. We therefore use the explicit algebraization by the Weierstrass function for elliptic curves to be able to work out the explicit Galois action on the associated hermitian $R$-lattices (see Proposition \ref{lemration}) and we then extend it, component by component, to products of elliptic curves in order to obtain our main results (Theorems~\ref{KQ} and \ref{QK}). Notice that the main difficulty is to be able to keep the abelian varieties, their isogenies \emph{and} their analytic representation defined over $\overline{\QQ}$ to be able to translate the action of $\Gal(\overline{\QQ}/\QQ)$.

In Section~\ref{sect4}, we look at the case when $\QQ$ is the field of moduli of the indecomposable principally polarized abelian varieties $A\simeq E_1 \times \cdots \times E_g$ with $E_i$ with CM by $R$. We first extend the result of \cite{GHR} showing that if $\QQ$ is the field of moduli then $\Cl(R)$ has exponent dividing $g$. We also show that the Steinitz class of the $R$-lattice associated with $A$ is of order at most $2$ (see Proposition \ref{exponent}). From this, we deduce the surprising Corollary~\ref{corodd} that odd-dimensional $A$ with field of moduli $\QQ$ must be isomorphic to the power of an elliptic curve. At the end of Section~\ref{sect4} we present the computations using the algorithms previously developed. In Table \ref{g2} we give the complete classification of all principally polarized indecomposable abelian surfaces with field of moduli $\QQ$ sorted by discriminants. In Table \ref{g3} we give a similar classification for abelian varieties of dimension $g=3.$ The classifications are complete thanks to the results of Appendix \ref{appA} written by Francesc Fité and Xavier Guitart. Proposition \ref{exponent}, shows that we should consider discriminants of quadratic fields with class group exponent $g$ where $g$ is the dimension of the abelian varieties. This condition is too weak to have an unconditional classification. Indeed, even if it is known that there are only finitely many discriminants of exponent $2$ and $3$, the complete list is known only under the Extended Riemann Hypothesis (see \cite{EKN}). Proposition \ref{prop:g2} shows that a dimension $2$ split Jacobian $(E_1\times E_2)$ with field of moduli $\QQ$ and $E_i$ with CM by the same field $K$ must satisfy $\#\Cl(K)\in\left\{1,2,4\right\}$. There exist quadratic imaginary fields with class group exponent $2$ of order $8$ and $16$ (and no other according to ERH) but the proposition guarantees that we will not find any curve corresponding to these discriminants. Proposition \ref{prop:g3} gives a similar bound in dimension $g=3$ ; $\#\Cl(K)\in\left\{1,3\right\}$. It is interesting to note also that we can run the computations for the exponent $2$ discriminant of class number $8$ and $16$ and, unsurprisingly, we find no Jacobian with field of moduli $\QQ.$ However, the computations are too expensive for the discriminant $\Delta=-4027$ with class group $\Cl(\Delta)\simeq \left(\ZZ/3\ZZ\right)^2$. Hence, even under ERH, the classification would not have been complete without the Appendix \ref{appA}.

\subsection*{Acknowledgments}

I would like to acknowledge Francesc Fit\'e and Xavier Guitart who motivated this project and helped until its fulfillment. As mentioned before, they also wrote the Appendix \ref{appA} which allows the Extended Riemann Hypothesis to be dropped in the classifications. Thank you also for meaningful discussions we had at the COUNT conference.

In addition, I want to thank Marco Streng for his attention and advice along the way of this work.

I would also like to address a special thank to Markus Kirschmer for his advice and patience answering all of my questions about hermitian lattices. He also gave me an extension of the Magma library of \cite{KNRR} to handle more efficiently the classification of free hermitian unimodular $R$-lattices which we used for the computation of the dimension $g=3$ case in Section \ref{g23}.

Finally, I want to thank Harun Kir for the rich discussions we had on this subject and for his precious support for this work.

\section{Complex abelian varieties and complex tori}\label{sect1}

\subsection{Abelian varieties and polarizable tori}

Let us first recall that there is an equivalence of categories between abelian varieties over $\CC$ and the category of polarizable tori given by $\T\colon A\mapsto A(\CC)$, see \cite[Theorem~2.9]{Mil}. A \textit{complex torus} is a quotient of groups $X=V/\Gamma$ with $V$ a complex vector space of finite dimension and $\Gamma$ a $\ZZ$-\textit{lattice} of $V$, i.e., a subgroup of $V$ generating $V$ over the reals. A torus $X=V/\Gamma$, or a lattice $\Gamma\subset V$, is said to be \textit{polarizable} if there exists a positive definite hermitian form $h\colon V\times V\rightarrow \CC$ such that
$$ \im h(\Gamma,\Gamma)\subset \ZZ$$
where $ \im$ is the imaginary part. Morphisms of complex tori are, in particular, morphisms of groups $\varphi\colon X=V/\Gamma\rightarrow X'=V'/\Gamma'$. Such a map can be lifted to a $\CC$-linear map $\varphi\an\colon V\rightarrow V'$ that sends $\Gamma$ to $\Gamma'$ called the \textit{analytic representation} of $\varphi$. The group morphism $\varphi\rat = {\varphi\an}_{|\Gamma}\colon \Gamma\rightarrow\Gamma'$ is called the \textit{rational representation} of $\varphi$. We denote the set of the analytic representations of morphisms between $X$ and $X'$ by $\Hom_\CC(\Gamma,\Gamma')$. We will often consider analytic representations directly as morphisms of polarizable tori.

A surjective morphism $f\colon A\rightarrow B$ between abelian varieties $A$ and $B$ over $\CC$ of equal dimension is called an \textit{isogeny}; its kernel is finite and its cardinality $\# \ker f$ is called the \textit{degree} of $f$, denoted $\deg f$. Via the functor $\T$, we will also call $\varphi \colon V/\Gamma \longrightarrow V/\Gamma'$ an isogeny. The degree of $\varphi$ is defined in the same way as $\deg\varphi=\#\ker \varphi$. Moreover, it satisfies
$$\deg \varphi=\bracket{\Gamma'\colon \varphi\an(\Gamma)}=\deg f$$
with $\bracket{\Gamma'\colon \varphi\an(\Gamma)}=\# (\Gamma'/\varphi\an(\Gamma))$, the index of $\varphi\an(\Gamma)$ in $\Gamma'.$

\subsection{Duality}

Let $V$ be a $g$-dimensional complex vector space and $\Gamma\subset V$ a $\ZZ$-lattice. The dual lattice associated with $\Gamma$ is defined by $$\widehat{\Gamma}=\left\{\ell\in V^*\mid  \im \ell(\Gamma)\subset \ZZ\right\}$$
where $V^*$ denotes the set of antilinear forms $\ell\colon V\rightarrow \CC$. This defines the dual complex torus of $V/\Gamma$, denoted $\widehat{V/\Gamma}=V^*/\widehat{\Gamma}$.\\
If $A$ is a complex abelian variety such that $A(\CC)\simeq V/\Gamma$ then there exists an isomorphism $\widehat{A}(\CC)\simeq \widehat{V/\Gamma}$, where $\widehat{A}$ is the dual abelian variety of $A$ (see \cite{Mum}). Moreover, for a morphism $f\colon A\rightarrow B$ and $\varphi=\T(f)\colon V/\Gamma\rightarrow V'/\Gamma'$ we have $\T(\widehat{f})\an=\varphi\an^*$ with
$$ \begin{array}{rrl}
 \varphi\an^*\colon &V'^* &\longrightarrow V^*\\
  &  \ell &\longmapsto \ell\circ \varphi\an
\end{array}$$
where $\widehat{f}\colon \widehat{B}\rightarrow \widehat{A}$ is the dual morphism of $f$ (see \cite[Section~2.4]{Bir}).

\subsection{Polarizations and hermitian forms}

Let $A$ be an abelian variety and let $\mathcal{L}$ be a line bundle. We consider the map
$$ \begin{array}{rrl}
 a_\mathcal{L}\colon &A &\longrightarrow \widehat{A}\\
  &  x&\longmapsto t_x^*\mathcal{L}\otimes\mathcal{L}^{-1}.
\end{array}$$
with $t_x$ the translation by $x$ map. Consider isomorphisms $A(\CC)\simeq V/\Gamma$ and $\widehat{A}(\CC)\simeq V^*/\widehat{\Gamma}$. Let $h=c_1(\mathcal{L})$ be the first Chern class of $\mathcal{L}$, which we identify with a hermitian form on $V$ (see \cite[Lemma~2.4.5]{Bir}). Then the map $\rho_h\in \Hom_\CC(\Gamma,\widehat{\Gamma})$ defined by
$$ \begin{array}{rrl}
 \rho_h\colon &V &\longrightarrow V^*\\
  &  v&\longmapsto h(v,\_).
\end{array}$$
is the analytic representation of $a_\mathcal{L}$. A polarization on an abelian variety $A$ is an isogeny $a_\mathcal{L}$ with $\mathcal{L}$ an ample line bundle. In this case, its first Chern class $h$ is a positive definite hermitian form (see \cite[Proposition 4.5.2]{Bir}).

A pair $(A,a)$ with $A$ an abelian variety and $a$ a polarization is called a \textit{polarized abelian variety}. Morphisms of polarized abelian varieties are defined by maps $f\colon (A,a)\rightarrow (B,b)$ such that $\widehat{f}bf=a$  and we call them \textit{polarized isogenies}. The distinction between polarized and non-polarized isogenies is essential; when we do not specify that an isogeny between polarized varieties is itself polarized, it refers to a morphism between the underlying abelian varieties without their polarizations. A \textit{polarized torus} is a pair $(V/\Gamma,\rho_h)$, also denoted by $(\Gamma,h)$, with $V/\Gamma$ a complex torus and $\rho_h\in\Hom_\CC(\Gamma,\widehat{\Gamma})$ induced by a positive definite hermitian form $h$, i.e., $\rho_h(v)=h(v,\_)$. We analogously define morphisms $\varphi\colon (V/\Gamma,\rho_h)\rightarrow (V'/\Gamma',\rho_{h'})$ between polarized tori. Their analytic representation must satisfy $\varphi_{\text{an}}^*\rho_{h'}\varphi_{\text{an}}=\rho_h$ which means that for $v,w\in V$, $$\rho_h(v)(w)=h(v,w)=h'(\varphi_{\text{an}}(v),\varphi_{\text{an}}(w)).$$ In particular, analytic representations of polarized isogenies define isometries on the associated hermitian vector spaces. In the following we will call a pair $(V,h)$ made of a $\CC$-vector space and $h$ a positive definite hermitian form a \emph{hermitian space}. Since the functor $\T$ is an equivalence of categories, it also defines an equivalence of categories $\PT$ between polarized abelian varieties and polarized tori $(A,a)\mapsto(X=\T(A),\rho_h=\T(a))$.

For any group schemes $A$ and $B$ over a field $k$ there is an isomorphism of groups
$$A(k)\times B(k)\simeq (A\times_k B)(k).$$

We can apply it to abelian varieties over $\CC$. For $A(\CC)\simeq V_A/\Gamma_A$, $B(\CC)\simeq V_B/\Gamma_B$ and $(A\times B)(\CC)\simeq V/\Gamma$ we have $$V/\Gamma\simeq (V_A/\Gamma_A)\times (V_B/\Gamma_B).$$
Hence $\Gamma\simeq \Gamma_A\times\Gamma_B$ and in terms of categories, the functor $\T$ commutes with products. In the same way, the functor $\PT$ also commutes with products.

\section{Totally split CM abelian varieties and $R$-module structure}\label{sect2}
Let $R$ be the maximal order of a quadratic imaginary field $K=\QQ(\sqrt{-d})$ with $d$ a positive square-free integer. An $R$-lattice is a finitely presented, torsion-free $R$-module. We denote by $\RLat$ the category of $R$-lattices.
In this section we want to show that there exist equivalences of categories
\begin{itemize}
    \item between $\AVR$, the category of complex abelian varieties isomorphic to a product of elliptic curves with CM by $R$, and $\RLat$ ;
    \item between $\PAV$, the category of complex polarized abelian varieties isomorphic to a product of elliptic curves with CM by $R$, and $\RhLat$, the category of integral hermitian $R$-lattices, i.e. $R$-lattices $L$ equipped with a positive definite hermitian form $H$ on $KL=L\otimes_R K$, such that $H(L,L)\subset R$.
\end{itemize}
For the rest of the article we fix field extensions
$$\QQ\longrightarrow K\longrightarrow \overline{\QQ}\longrightarrow\CC.$$
\subsection{Elliptic curve with complex multiplication over $\CC$}\label{sec2.1}

Let $R=\ZZ[\omega]$ be the maximal order of an imaginary quadratic field $K=\QQ(\sqrt{\Delta})$ with $\Delta<0$ the discriminant of $K/\QQ$ and $\alpha_R=\im \omega >0$. Notice that $\alpha_R$ does not depend on the chosen generator $\omega$ with positive imaginary part.

Let $E$ be an elliptic curve over $\CC$ with complex multiplication by $R$, i.e., there exists a ring isomorphism $\End(E)\simeq R$. Let $\Lambda\subset \CC$ be a lattice such that there is an isomorphism $\eta\colon  \CC/\Lambda \rightarrow E(\CC)$ (of Lie groups). Then we will denote $E$ by $E_\Lambda$. By \cite[Proposition II.1.1.]{Sil} there is a unique isomorphism $$[.]_E\colon R\rightarrow \End(E)$$
characterized, for any $\alpha\in R,$ by the commutativity of the diagram of Figure \ref{FigRiso}.
\begin{figure}[ht]
\[
\xymatrix{\CC/\Lambda \ar[r]^{z\mapsto\alpha z}\ar[d]_{\eta}&\CC/\Lambda  \ar[d]^{\eta} \\
E_\Lambda(\CC) \ar[r]^{[\alpha]} & E_\Lambda(\CC) .}
\]
\caption{The bracket isomorphism}\label{FigRiso}
\end{figure}
\newline
This isomorphism does not depend on $\eta$ and we will use it when we identify $R$ with $\End(E).$

According to \cite[Proposition~2.1]{Sil} all CM elliptic curves over $\CC$ admit a model over $\overline{\QQ}.$ We denote by $\Ell(R)$ the set of isomorphism classes of elliptic curves over $\overline{\QQ}$ with CM by $R.$ We recall from \cite[Proposition~1.2]{Sil} that $\Cl(R)$, the class group of $R,$ acts simply transitively on $\Ell(R)$ and the action is given by $$\fa\star E_\Lambda=E_{\fa^{-1}\Lambda},$$
where $\fa$ is a fractional ideal of $R.$ The action is well defined on the isomorphism classes because another representative $\fb$ of the class of $\fa$ in $\Cl(R)$ differs from $\fa$ by a scalar and then $\fa^{-1}\Lambda$ and $\fb^{-1}\Lambda$ are homothetic lattices.

\subsection{Totally split complex tori and $R$-lattice structure}
Let $L$ be an $R$-lattice, i.e., a finitely presented, torsion-free $R$-module. Since $R$ is maximal, $L$ is a module over a Dedekind domain and by \cite[Theorem 81.3]{OMe}, we can always write $L$ as a sum $$L=\bigoplus_{i=1}^g \fa_i x_i,$$ with $(x_1,\dots,x_g)$ a basis of $KL$ and $\fa_1,\dots,\fa_g$ fractional ideals of $R$. The family $(\fa_i,x_i)$ is called a \textit{pseudo-basis} of $L$ and the \textit{Steinitz class} $\st(L)$ of $L$ is defined by the class of the product $\fa_1\cdots\fa_g$ in $\Cl(R)$. The Steinitz class of a lattice together with its rank determines its $R$-isomorphism class by \cite[Theorem~13]{KCon2}.\\
For an $R$-ideal $\fa$, the norm of $\fa$ is defined by $N(\fa)=\# \left( R/\fa\right)$. We can extend the definition of the norm to any fractional ideal by the equality $N(\lambda\fa)=|\lambda|^2N(\fa)$ for any $\lambda\in K$.

Throughout this article, different notations will be used to denote lattices depending on the context:
\begin{itemize}
    \item The notation $\Gamma$ will denote a lattice in a $\CC$-vector space.
    \item The notation $\Lambda$ will denote a lattice in $\CC$ only.
    \item The notation $L$ will denote a lattice over an imaginary quadratic order $R$.
\end{itemize}
Thus, the notations $\Gamma$ and $\Lambda$ will be reserved for the analytical aspect, for complex tori, while $L$ will be used for the algebraic aspect. The functor $\R$, defined in Theorem \ref{theormod}, connects complex tori to hermitian lattices, and these notations will help clarify the category in which we are working. For the same reasons we write $h$ for the hermitian forms coming from polarizable tori and $H$ for $R$-lattices.
\begin{theo}\label{theormod}
Let $R$ be an order in a quadratic imaginary field $K$. There is an equivalence of categories between the category $\AVR$ of abelian varieties isomorphic to a product of elliptic curves with CM by $R$, and the category $\RLat$ of $R$-lattices given by
$$\begin{array}{rcl}
 \R\colon &\AVR &\rightarrow \RLat\\
  & A \text{ s.t. } A(\CC)\simeq V/\Gamma &\mapsto \Gamma\\
  & (\CC L)/L& \mapsfrom L.
  \end{array}$$
  on objects and, for arrows, $f\colon A(\CC)\simeq V/\Gamma\rightarrow A'(\CC)\simeq V'/\Gamma',$ given by $\R(f)=f\rat.$ 
\end{theo}
\begin{proof}
Let $A\in \AVR$ such that $\T(A)=A(\CC)=V/\Gamma.$ There is an isomorphism $\varphi\colon V/\Gamma\rightarrow \CC^g/\bigoplus \Lambda_i.$ The $\ZZ$-lattices $\Lambda_i$ have a natural structure of and $R=\End(\Lambda_i)$-module. The lattice $\Gamma$ is stable by multiplication by $R$. Indeed,
$$\varphi\an(R\Gamma)=R\varphi\an(\Gamma)=R\bigoplus_i\Lambda_i=\bigoplus_i \Lambda_i=\varphi\an(\Gamma)$$ so $R\Gamma=\Gamma.$
Moreover, the isomorphism $\varphi$ endows $\Gamma$ with a structure of an $R$-module in a natural way \begin{equation}\label{eq1}
    r.\gamma = \varphi\an^{-1}(r\varphi\an(\gamma))=r\gamma, \text{ for } r\in R,\gamma\in\Gamma.
\end{equation}
Hence $\Gamma$ has a structure of $R$-module and we can see in ($\ref{eq1}$) that this structure does not depend on the isomorphism $\varphi$ we chose.

Conversely, let $L$ be an $R$-lattice and $(\fa_i,x_i)$ a pseudo-basis of it. Since the $\fa_i$ are fractional ideals, there exists an integer $n$ such that $n\fa_i\subset R$ for all $i,$ so the multiplication by $n$ map defines an isogeny $\Hom_\CC(L,R^g).$ If we endow $R^g$ with the canonical hermitian form $h_0(x,y)=\trsp x\overline{y}$, we have $h_0(R^g,R^g)=R$, hence $$\im h_0(R^g,R^g)=(\im \omega)\ZZ=\alpha_R\ZZ.$$

Thus, $\left(R^g,\frac{1}{\alpha_R}h_0\right)$ defines a polarized torus and so does $\left(L,\frac{n^2}{\alpha_R}h_0\right)$. This proves that the underlying $\ZZ$-lattice $\Gamma$ of an $R$-lattice $L$ is polarizable.

Finally, given a morphism $f\colon V/\Gamma\rightarrow V'/\Gamma',\R(f)=f\rat\colon \Gamma\rightarrow \Gamma'$. Since $f\rat={f\an}_{|\Gamma}$ and $f\an$ is $\CC$-linear it is then $R$-linear. Hence $\R$ maps arrows in a full and faithful way.

Hence there is an equivalence of categories between $R$-lattices and polarizable tori $X$ such that there exists an isomorphism $X\rightarrow \CC^g/\bigoplus\Lambda_i.$ The latter being in equivalence with the category of abelian varieties isomorphic to a power of elliptic curves with CM by $R.$ This proves the equivalence of categories between $\AVR$ and $\RLat$.
\end{proof}
It may be worth elaborating why the CM case is so special. If we consider the functor $\R$ from abelian varieties over $\CC$ without restriction to the category of $\ZZ$-lattices of a finite dimensional $\CC$-vector space then $\R$ is not essentially surjective. Indeed, some $\ZZ$-lattices $\Gamma$ of $V$ of dimension greater than $2$ are not polarizable. Even if we restrict $\R$ on its essential image it is not full. Indeed, even in dimension $1$ there are morphisms of polarizable $\ZZ$-lattices (i.e., morphisms of groups) which are not the restriction of a $\CC$-linear map.

Hence the structure of a $\ZZ$-module is not enough. Fortunately, considering the $R$-module structure given by the complex multiplication makes $\R$ an equivalence.

\subsection{Polarizable tori and integral lattices}\label{poltori}

A \textit{hermitian $R$-lattice} is defined as a pair $(L,H)$ with $L$ an $R$-lattice and $H$ a positive definite hermitian form on the ambient space $KL$. The \textit{scale} of a hermitian lattice $(L,H)$ is defined as the set $\fs(L)=H(L,L)\subset K$. It is a fractional ideal of $K$ (see \cite[Remark 2.3.4]{Kir16}). The \emph{dual lattice} $L^\#$ of a hermitian lattice is the lattice defined by $L^\#=\left\{v\in KL\mid H(v,L)\subset R\right\}$. For $\fa$ a fractional ideal of $K,$ we say that $(L,H)$ is \textit{$\fa$-modular} if $\fa L^\#=L$. If $(L,H)$ is $\fa$-modular then its scale satisfies $\fs(L)=\fa$. A hermitian $R$-lattice $(L,H)$ is said to be \textit{integral} if
$$H(L,L)\subset R,$$
i.e., it is integral if its scale is an integral ideal (or equivalently $L\subset L^\#$). An $R$-modular hermitian lattice $(L,H)$ is called unimodular (this is equivalent to $(L,H)$ being integral and its scale being $\fs(L) = R$). One also defines the \textit{volume} of a hermitian lattice $(L,H)$ as the fractional ideal
$$\fv(L)=\left(\prod_{i=1}^g N(\fa_i)\right)\det(G(x_1,\dots,x_g))R$$
where $G(x_1,\dots,x_g)=(H(x_i,x_j))_{1\leq i,j\leq g},$ the Gram matrix of $(x_1,\dots,x_g).$ A hermitian lattice $(L,H)$ is $\fa$-modular if, and only if,
$$\fv(L)=\fa^g\text{ and } \fs(L)=\fa.$$
(see \cite[Section~2]{Hof}).

In this section we explain the link between polarized tori and integral lattices.

Let $(V,H)$ be a hermitian $\CC$-vector space and let $\alpha\in\RR_{>0}$. We denote by $V^\alpha$ the vector space $V$ provided with the hermitian form $H^\alpha(x,y)=\alpha H(x,y)$. For a lattice $L$ in $(V,H)$ we denote by $L^\alpha$ the hermitian lattice $L$ regarded in the hermitian space $(V^\alpha,H^\alpha)$ as in \cite[Section~82J.]{OMe}.

\begin{lem}\label{integ}
Let $R=\ZZ[\omega]$ be an order of an imaginary quadratic field and $\fa$ be a sub-$R$-module in $\CC$. Then $\fa$ is an integral ideal in $R$ if, and only if, $\im \fa\subset \alpha_R \ZZ$ with $\alpha_R=\im(\omega)$.
\end{lem}
\begin{proof} The left-to-right direction is straightforward.

Let $a=x+y\omega\in\fa$ with $x,y\in\RR$. Since $\im a=y\alpha_R\in\alpha_R\ZZ$ we have $y\in \ZZ$. Moreover, $\overline{\omega}\in R$ so $a\overline{\omega}\in \fa$ and $\im a\overline{\omega}=-x\alpha_R$ so $x\in\ZZ$. Hence $a\in R$ and therefore $\fa\subset R$.
\end{proof}
\begin{prop}\label{propscale}
Let $(V/\Gamma,\rho_h)$ be a polarized torus such that $\Gamma\simeq \bigoplus_{i=1}^g \Lambda_i$ with $\End_\CC(\Lambda_i)\simeq R$ with $R=\ZZ[\omega]$. Then $\fs(\Gamma^{\alpha_R})$ is an integral ideal of $R$. Moreover, we have the equality
$$\deg \rho_h=\bracket{(\Gamma^{\alpha_R})^\#\colon\Gamma^{\alpha_R}}.$$
\end{prop}
\begin{proof}
We know that $h(\Gamma,\Gamma)$ is a $R$-module in $\CC$ and since $(\Gamma,h)$ is a polarizable torus we must have $\im h(\Gamma,\Gamma)\subset \ZZ$. Hence by Lemma \ref{integ}, we have $\fs(\Gamma^{\alpha_R})=\alpha_R h(\Gamma,\Gamma)\subset R$.
Moreover,
\begin{align*}
    \deg \rho_h&=\bracket{\widehat{\Gamma}\colon \rho_h(\Gamma)}&\\
    &=\bracket{\rho_h^{-1}(\widehat{\Gamma})\colon\Gamma}&\\
    &= \#\left\{v\in V\mid \im h(v,\Gamma)\subset \ZZ\right\}/\Gamma&\\
    &= \#\left\{v\in V\mid (\im\omega)h(v,\Gamma)\subset R\right\}/\Gamma &\text{(by Lemma \ref{integ})}\\
    &= \# ((\Gamma^{\alpha_R})^\#/\Gamma)&\\
    &= \bracket{(\Gamma^{\alpha_R})^\#\colon\Gamma^{\alpha_R}}.&
\end{align*}
\end{proof}
\comm{
\begin{rem}
If moreover the lattice $(\Gamma^{\alpha_R},h^{\alpha_R})$ is $\fs(\Gamma^{\alpha_R})$-modular then
$$\deg\rho_h=\bracket{(\Gamma^{\alpha_R})^\#\colon\fs(\Gamma^{\alpha_R}}.(\Gamma^{\alpha_R})^\#)=N(\fs(\Gamma^{\alpha_R}))^g.$$
Hence classes of principally polarized tori $(\Gamma,h)$, i.e., with $\deg\rho_h=1$ correspond to isometry classes of integral lattices whose scale has norm $1$ and it is an integral ideal by Proposition \ref{propscale}. So it corresponds to unimodular lattices.
\end{rem}}
Let $\Lambda\subset \CC$ be a lattice such that $\End(\Lambda)=R.$ We define $\PTR$ the subcategory of polarized tori $(X=V/\Gamma,\rho_h)$, with $\Gamma\simeq \bigoplus \Lambda_i$ and $\End(\Lambda_i)\simeq R$. We conclude the section with the following theorem.
\begin{theo}\label{equivR}
 With the notation above, there is an equivalence of categories given on objects by
$$ \begin{array}{rrl}
  &\PTR &\rightarrow \RhLat\\
  & (X=V/\Gamma,\rho_h) &\mapsto (\Gamma^{\alpha_R},h^{\alpha_R})\\
  & ((\CC L)/L,H^{1/\alpha_R}) &\mapsfrom (L,H).
\end{array}$$
\end{theo}
Since $\PTR$ is equivalent to the category $\PAV$ via the functor $\T$, we have the equivalence between $\PAV$ and $\RhLat.$ We call this functor $\Rh$ and note that it satisfies $$ \begin{array}{rrl}
 \Rh\colon &\PAV &\rightarrow \RhLat\\
  & (A,a) &\mapsto (\R(A)^{\alpha_R},\R(a)^{\alpha_R}).
  \end{array}$$
Since every elliptic curve over $\CC$ with CM has a model over $\overline{\QQ}$, the category of polarized abelian varieties over $\Qbar$ isomorphic to a product of elliptic curves $E_i$ with CM by $R$ is equivalent to $\PAV$ by the base change functor
$$A\mapsto A_\CC.$$
Note that morphisms over $\CC$ of abelian varieties over $\Qbar$ isomorphic to a product of CM elliptic curves are actually defined over $\Qbar$ by \cite[Theorem~2.2.(c)]{Sil}.

\section{Action of $\GalQQ$}\label{sect3}

The \textit{field of moduli} of a polarized abelian variety $(A,a)$ over $\Qbar$ is the field fixed by the subgroup
$$\left\{\sigma\in\GalQQ\mid (A^\sigma,a^\sigma)\simeq (A,a)\right\}.$$

Given a maximal order in an imaginary quadratic field $K$ we would like to construct an algorithm to enumerate all the isomorphism classes of $\PAV$ which have field of moduli $\QQ.$ In order to do this we first need to understand the action of $\GalQQ$ on the hermitian lattices through the functor $\Rh$ described in Section \ref{sect2}. In other words, given $(A,a)\in\PAV$ and $\sigma\in \GalQQ$ we want to understand the isometry class of $\Rh(A^\sigma,a^\sigma)$ in terms of $\Rh(A,a)$ and $\sigma.$

In order to do so, we need to understand the action of $\GalQK$ and the action of $\Gal(K/\QQ)$, by complex conjugation, separately in order to recover the action of $\GalQQ\simeq \GalQK\rtimes \Gal(K/\QQ).$ We describe the action of $\Gal(K/\QQ)$ and $\GalQK$ in Theorem \ref{KQ} and Theorem \ref{QK} respectively. The end of this section is devoted to their proofs.

Let us first introduce the necessary notation.

Let $(L,H)$ be a hermitian lattice. Let $\iota\colon (KL,H)\rightarrow (K^g,H')$ be an isometry. We define $(\overline{L},\overline{H})$ by $\overline{L}=\iota^{-1}\overline{\iota(L)}$ and $$\overline{H}(x,y)=\overline{H(\iota^{-1}\overline{\iota(x)},\iota^{-1}\overline{\iota(y)})}=\overline{H'(\overline{\iota(x)},\overline{\iota(y)})}$$
where $\overline{\cdot}$ refers to the complex conjugation which is the unique non-trivial automorphism of $\Gal(K/\QQ).$
The isometry class of $(\overline{L},\overline{H})$ is independent of the choice of $\iota$. We can now describe the action of the complex conjugation.
\begin{theo}[Description of the action of $\Gal(K/\QQ)$]\label{KQ}
Let $(A,a)\in\PAV$ considered over $\Qbar$ and let $\Rh(A,a)=(L,H)$. Then there is an isometry
$$\Rh\left(\overline{A},\overline{a}\right)\simeq \left(\overline{L},\overline{H}\right).$$
\end{theo}
Let $E$ be an elliptic curve defined over $\Qbar$ with CM by $R.$ Let $\Lambda\subset \CC$ be a lattice such that $E_\CC\simeq E_\Lambda$. We will often abuse notation slightly and also denote by $E$ the base change $E_\CC$ of $E$, to avoid heavy notation such as $E_\CC(\CC)$. Recall that for $\dot{\fa}\in\Cl(R),$ the elliptic curve $\dot{\fa}\star E_\Lambda$ is defined by $E_{\fa^{-1}\Lambda}.$ By \cite[Proposition~2.4]{Sil}, there is a surjective group morphism
\begin{equation}\label{eq:F}
    F\colon \GalQK\rightarrow \Cl(R)
\end{equation}
such that the elliptic curves $E^\sigma$ and $F(\sigma)\star E$ are isomorphic. We can now state the description of the action of $\GalQK$.
\begin{theo}[Description of the action of $\GalQK$]\label{QK}
Let $(A,a)\in\PAV$ considered over $\overline{\QQ}$. Let $\Rh(A,a)=(L,H),\sigma\in\GalQK$ and let $F(\sigma)=\dot{\fa}^{-1}\in\Cl(R)$. Then there is an isometry
$$\Rh(A^\sigma,a^\sigma)\simeq \left(\fa L,\frac{1}{N(\fa)}H\right),$$
where $N(\fa)$ is the norm of $\fa.$
\end{theo}

\subsection{Positioning of the problem }
Let us recall from Section \ref{sec2.1} that for every elliptic curve with CM by $R$ and any isomorphism $\eta\colon \CC/\Lambda\rightarrow E(\CC)$ there is a unique isomorphism $[\cdot]_E\colon R\rightarrow \End(E)$ such that the following diagram commutes
\[
\xymatrix{\CC/\Lambda \ar[r]^{\alpha}\ar[d]_{\eta}&\CC/\Lambda  \ar[d]_{\eta}\\
E_\Lambda(\CC) \ar[r]^{[\alpha]} & E_\Lambda(\CC).}
\]
By \cite[Theorem~2.2]{Sil} the bracket isomorphism satisfies the following
$$\left([\alpha]_E\right)^\sigma=[\alpha^\sigma]_{E^\sigma}\text{ for all }\sigma\in\Aut(\CC),\alpha\in R.$$
This is true for any $\Lambda$ and $\eta\colon \CC/\Lambda\xrightarrow{\sim} E(\CC)$ we choose and for any lattice identification $\eta_\sigma\colon\CC/\Lambda_\sigma\xrightarrow{\sim} E^\sigma(\CC)$ of $E^\sigma$ as long as we choose the same $\eta_\sigma$ on both sides of the diagram. Of course if we do not take the same isomorphisms on both sides of the diagram, for instance we chose $\eta$ and $-\eta$, this property does not hold anymore. The main difficulty we will encounter is that we want to deal with isogenies $f\colon E\rightarrow E'$ between possibly non-isomorphic elliptic curves. If we expect diagrams like
\[
\xymatrix{\CC/\Lambda \ar[r]^{\alpha}\ar[d]_{\eta}&\CC/\Lambda'  \ar[d]^{\eta'} \\
E(\CC) \ar[r]^{f} & E'(\CC).}
\]
to have a nice behavior with respect to $\Aut(\CC)$ or $\GalQQ$ we need a kind of canonical way of describing the action of $\Aut(\CC)$ (or $\GalQQ$) on lattices, and a canonical way of choosing the isomorphisms $\eta$ (to avoid the $-\eta$ issue for instance). We will show that the $\wp$-functions of Weierstrass will do the job.

First we need to specify the isomorphisms we refer to when we consider the analytic representation of a morphism between abelian varieties. Let $A$ and $A'$ be abelian varieties over $\CC$. Consider isomorphisms $\eta\colon V/\Gamma\xrightarrow{\sim} A(\CC)$ and $\eta'\colon V'/\Lambda'\xrightarrow{\sim} A'(\CC)$. Every morphism $f\colon A\rightarrow A'$ induces a morphism $\varphi$ such that the diagram in Figure \ref{anarep} commutes.
\begin{figure}[ht]
\[
\xymatrix{V/\Gamma \ar[r]^{\varphi}\ar[d]_{\eta}&V'/\Gamma'  \ar^{\eta'}[d]\\
A(\CC) \ar[r]^{f} & A'(\CC) .}
\]
\caption{Analytic representation of an isogeny}\label{anarep}
\end{figure}
\newline
The morphism $\varphi$ of tori can be lifted to a linear map $\beta\colon V\rightarrow V'$ such that $\beta(\Gamma)\subset \Gamma'$. We call $\beta$ the analytic representation of $f$ associated with the isomorphisms $\eta$ and $\eta'$ or simply the analytic representation of $(f,\eta,\eta').$

We will use the Weierstrass $\wp$-functions and Eisenstein series as in \cite{Sil1} as a canonical way of identifying an elliptic curve over $\CC$ with a complex torus and we will study the field of definition of the induced analytic representation of isogenies between CM elliptic curves in Section \ref{sectelliptic}. Then we will extend the results for elliptic curves to product of elliptic curves, component by component, in Section \ref{sectprod} and finally prove Theorem \ref{KQ} and \ref{QK}.

\subsection{Action of $\GalQQ$ on CM elliptic curves}\label{sectelliptic}

\subsubsection{Fixing isomorphisms with Weierstrass $\wp$-functions}

By \cite[Therorem~5.1]{Sil1}, for any $A,B\in\CC$ such that $4A^3-27B^2\neq 0$ there exists a unique lattice $\Lambda\subset \CC$ such that $$A=g_2(\Lambda)=60\sum_{\omega\in\Lambda\setminus\{0\}}\frac{1}{\omega^4}\text{ and }B=g_3(\Lambda)=140\sum_{\omega\in\Lambda\setminus\{0\}}\frac{1}{\omega^6}$$
where $g_2(\Lambda)$ and $g_3(\Lambda)$ are Eisenstein series of weight $4$ and $6$ of $\Lambda,$ respectively. Let $E/\CC$ be the elliptic curve defined by the Weierstrass model $E\colon y^2=4x^3-Ax-B.$ According to \cite[Proposition~3.6]{Sil1} there is an isomorphism of Lie groups
\begin{equation}\label{philambda}
    \begin{array}{rrl}
        \phi_\Lambda\colon & \CC/\Lambda\xrightarrow{} & E_\Lambda(\CC)\subset \mathbb{P}^2(\CC)  \\
        & z\mapsto & [\wp(z,\Lambda)\colon\wp'(z,\Lambda)\colon 1]
    \end{array}
\end{equation}
with $\wp$ the Weierstrass $\wp$-functions. Moreover, by \cite[VI.5.1]{Sil1}, every elliptic curve E over $\CC$ is isomorphic to some $E_\Lambda$ via $\phi_\Lambda$.

\emph{
From now on, the notation $E_\Lambda$ means that $E(\CC)\simeq \CC/\Lambda$ with the isomorphism given by $\phi_\Lambda.$
}\\

Let $\kappa\hookrightarrow \CC$ be a field extension and $E\colon y^2=4x^3-Ax-B$ with $A,B\in \kappa$ and let $\Lambda$ be a lattice such that $g_2(\Lambda)=A$ and $g_3(\Lambda)=B.$ For every $\sigma\in\Aut(\kappa)$ we define $\Lambda_\sigma$ as the unique lattice such that $g_2(\Lambda_\sigma)=A^\sigma$ and $g_3(\Lambda_\sigma)=B^\sigma$.

By definition of $\Lambda_\sigma$ we have an isomorphism $\phi_{\Lambda_\sigma}\colon\CC/\Lambda_\sigma\rightarrow E^\sigma(\CC)$ with $E^\sigma\colon y^2=4x^3-A^\sigma x-B^\sigma.$
\begin{lem}\label{compconj}
For any lattice $\Lambda\subset\CC,$ we have $\Lambda_{\hspace{0.05cm}\overline{\cdot}}=\overline{\Lambda}.$
\end{lem}
\begin{proof}
The Eisenstein series $g_2(\Lambda)$ and $g_3(\Lambda)$ are absolutely convergent and the complex conjugation is an antilinear map so
$g_k(\overline{\Lambda})=\overline{g_k(\Lambda)}.$
\end{proof}

\subsubsection{Algebraicity of analytic representations for CM elliptic curves}

The results here are simple consequences of the chapter of \cite[Chapter II]{Sil}. We simply state them in a form which will be adapted to our formalism later. Let $E,E'$ be two elliptic curves over $\Qbar$ by Weierstrass models. By \cite[Theorem 2.2.(c)]{Sil} an isogeny $f\colon E\rightarrow E'$ is also defined over $\Qbar$. In this section we want to study the analytic representation $\alpha$ of $(f,\phi_\Lambda,\phi_{\Lambda'})$ when $E$ and $E'$ have Weierstrass models over $\Qbar$, with $\phi_\Lambda\colon\CC/\Lambda\rightarrow E(\CC)$ and $\phi_{\Lambda'}\colon \CC/\Lambda'\rightarrow E'(\CC)$ as defined in (\ref{philambda}). In particular, we wish to know:
\begin{itemize}
    \item Is $\alpha$ in $\Qbar$ ?
    \item If $\alpha\in\Qbar$, does $\alpha$ behave well with $\GalQQ$, i.e., is $\alpha^\sigma$ the analytic representation of $(f^\sigma,\phi_{\Lambda_\sigma},\phi_{\Lambda'_\sigma})$ ?
\end{itemize}
We will positively answer these questions in Proposition \ref{lemration}. This is partially addressed in \cite[Chapter II]{Sil}, but we require greater precision and control over the analytic representations than the author does for the purposes of his book. Therefore, we will elaborate on this in the current section.
\begin{lem}\label{nur}
Let $\Lambda\subset \CC$ be a lattice and $E\colon y^2=4x^3-g_2(\Lambda)x-g_3(\Lambda)$ be the elliptic curve over $\CC$ associated with $\Lambda$ by $\phi_\Lambda$. Let $r\in\CC\setminus\set{0}$, let $E'$ be the elliptic curve associated with $r\Lambda$ and let $\nu_r$ be the map defined by
$$\begin{array}{rrl}
  \nu_r  & \PP^2(\CC)\longrightarrow & \PP^2(\CC)  \\
    & [x\colon y\colon 1] \longmapsto & [\frac{1}{r^2}x\colon\frac{1}{r^3}y\colon 1], 
\end{array}$$
 Then the restriction of $\nu_r$ to $E(\CC)\subset \PP^2(\CC)$ defines an isomorphism on its image $E'(\CC)$. Moreover, $r$ is the analytic representation of $({\nu_r}_{|E(\CC)},\phi_\Lambda,\phi_{r\Lambda})$.
\end{lem}
\begin{proof}
For all $z\in\CC\setminus r\Lambda, r\in\CC\setminus\{0\}, \wp(z,r\Lambda)=\frac{1}{r^2}\wp(\frac{z}{r},\Lambda)$ and $\wp'(z,r\Lambda)=\frac{1}{r^3}\wp'(\frac{z}{r},\Lambda)$. Hence
\begin{align*}
    \phi_{r\Lambda}(z)&=[\wp(z,r\Lambda)\colon\wp'(z,r\Lambda)\colon 1]\\
    &=[\frac{1}{r^2}\wp\left(\frac{z}{r},\Lambda\right)\colon \frac{1}{r^3}\wp'\left(\frac{z}{r},\Lambda\right)\colon 1]\\
    &=\nu_r\left(\phi_\Lambda\left(\frac{z}{r}\right)\right).
\end{align*}
This proves that the following diagram commutes
\[
\xymatrix{\CC/\Lambda \ar[r]^{r}\ar[d]_{\phi_\Lambda}&\CC/r\Lambda  \ar^{\phi_{r\Lambda}}[d]\\
E(\CC) \ar[r]^{{\nu_r}_{|E(\CC)}} & E'(\CC)}
\]
which shows that $({\nu_r}_{|E(\CC)},\phi_\Lambda,\phi_{r\Lambda})$ has analytic representation $r.$
\end{proof}

Let $\fa$ be a fractional ideal of $R$, and let $\Lambda$ be an $R$-lattice such that $E=E_\Lambda\in\AVR$. We defined $\fa\star E$ earlier as the isomorphism class of $E_{\fa^{-1}\Lambda}$.\\

\emph{\centering From now on we will refer to $\fa\star E$ as the elliptic curve defined by the Weierstrass equation $$\fa\star E\colon y^2=4x^3-g_2(\fa^{-1}\Lambda)x-g_3(\fa^{-1}\Lambda).$$}

Let $f\colon E\rightarrow E'$ be an isogeny between elliptic curves with CM by $R$ with $E'=E_{\Lambda'}$. 
Let $\fa$ be the integral ideal of $R$ such that $\ker f=E[\fa]=\cap_{a\in\fa} \ker[a]_E.$ 
According to \cite[Proposition~1.4]{Sil}, $\fa\star E\simeq E/E[\fa]$ and we can factor $f\colon E\rightarrow E'$ and $\alpha\colon \CC/\Lambda\rightarrow\CC/\Lambda'$ as in Figure \ref{factiso} where $\pi\colon \CC/\Lambda\rightarrow \CC/\fa^{-1}\Lambda$ is the natural projections and $p\colon E\rightarrow \fa\star E$ induced by $\pi$.
\begin{center}
   \begin{figure}[ht]
$$\begin{tikzcd}
\CC/\Lambda \ar[rr,"\alpha"]\ar[dd,"\phi_{\Lambda}"]\ar[rd,"\pi"]&   &\CC/\Lambda'  \ar[dd,"\phi_{\Lambda'}"]\\
    & \CC/\fa^{-1}\Lambda\ar[ru,"\alpha"]&   \\
E(\CC) \ar[rr,"f" near start] \ar[dr,"p"]&   & E'(\CC)\\
   &(\fa\star E)(\CC)\ar[ur,"\nu_\alpha"]\ar[from=uu,"\phi_{\fa^{-1}\Lambda}" near start, crossing over]. &   
\end{tikzcd}$$
\caption{Factorization of isogenies}\label{factiso}
\end{figure} 
\end{center}
\begin{lem}\label{tensorrat}
Let $\Lambda$ be a lattice with complex multiplication such that $g_k(\Lambda)\in\Qbar$ for $k=2,3.$ Then for any fractional ideal we have $\fa\subset K,g_k(\fa^{-1}\Lambda)\in\Qbar.$
\end{lem}
\begin{proof}
Let $E\colon y^2=4x^3-g_2(\Lambda)x-g_3(\Lambda)$ be defined over $\Qbar$. Since all CM elliptic curves have a model over $\Qbar$ we consider $E'\colon y^2=4x^3-A'x-B'$, a model of $\fa\star E$ over $\Qbar$ and $\Lambda'$ such that $E'=E_{\Lambda'}$. Consider any isogeny $f\colon E\rightarrow E'$ over $\Qbar$ and $\alpha\in\CC$ the analytic representation of $(f,\phi_\Lambda,\phi_{\Lambda'})$. Let $r\in\Qbar$ be the coefficient of the induced map on differentials $f^* \frac{dx'}{y'}=r \frac{dx}{y}$. By the proof of \cite[Proposition~3.6.(b)]{Sil1} we have $\phi_\Lambda^*\left(\frac{dx}{y}\right)=dz$. Thus, $$(f\circ\phi_\Lambda)^*\left(\frac{dx'}{y'}\right)=\phi_\Lambda^*\circ f^*\left(\frac{dx'}{y'}\right)=\phi_\Lambda^*\left(r \frac{dx}{y}\right)=r dz$$
   and 
   $$( \phi_{\Lambda'}\circ\alpha)^*\left(\frac{dx'}{y'}\right)=\alpha dz'=\alpha dz.$$
   Since the following diagram commutes
   \[\xymatrix{
   \CC/\Lambda\ar[r]^{\alpha}\ar[d]_{\phi_\Lambda} & \CC/\Lambda'\ar[d]^{\phi_{\Lambda'}}\\
   E(\CC)\ar[r]^{f} & E'(\CC)
   }
   \]
   we have equality of the differentials $r dz=\alpha dz$. Hence $r=\alpha\in\Qbar.$ Let $\fb$ be the integral ideal such that $\ker f=E[\fb].$ By Figure \ref{factiso} we have $\Lambda'=\alpha\fb^{-1}\Lambda$ so, for $k\in\set{2,3},$ $g_k(\fb^{-1}\Lambda)=\alpha^{2k}g_k(\Lambda')\in\Qbar$. Finally, $\fb^{-1}\Lambda$ and $\fa^{-1}\Lambda$ give isomorphic elliptic curves so they must be homothetic by some $\lambda\in\Hom_\CC(\fb^{-1}\Lambda,\fa^{-1}\Lambda)=\left\{\mu\in\CC\mid\mu\fb^{-1}\Lambda\subset \fa^{-1}\Lambda\right\}\subset K$. Hence $$g_k(\fa^{-1}\Lambda)=\lambda^{-2k}g_k(\fb^{-1}\Lambda)\in\Qbar.$$
\end{proof}
\begin{prop}
\label{lemration}
Let $E\colon y^2=4x^3-Ax-B$ over $\Qbar$ and $E'\colon y^2=4x^3-A'x-B'$ over $\CC$ be elliptic curves both with CM by $R$. Let $\Lambda$ and $\Lambda'$ be lattices with $g_2(\Lambda)=A, g_3(\Lambda)=B, g_2(\Lambda')=A'$ and $g_3(\Lambda')=B'$. Let $f\colon E\rightarrow E'$ be an isogeny over $\CC$. Consider $\alpha\in\CC$ the analytic representation of $(f,\phi_\Lambda,\phi_{\Lambda'})$. Then
\begin{enumerate}
    \item $\alpha\in\overline{\QQ}$ if, and only if, $g_2(\Lambda')\in\Qbar$ and $g_3(\Lambda')\in\Qbar$.
    \item If 1. is satisfied then for every $\sigma\in \Gal(\overline{\QQ}/\QQ), \alpha^\sigma$ is the analytic representation of $(f^\sigma,\phi_{\Lambda_\sigma},\phi_{\Lambda'_\sigma})$ with $f$ identified with the induced map over $\overline{\QQ}$.
    \item For every fractional ideal $\fa$ and for all $\sigma\in\GalQQ$ we have $\left(\fa\Lambda\right)_\sigma=\fa^\sigma\Lambda_\sigma.$
    
\end{enumerate}
\end{prop}
\begin{proof}
\begin{enumerate}
   \item Since $\Lambda'=\alpha\fa^{-1}\Lambda$ we have $g_k(\Lambda')=\alpha^{-2k}g_k(\fa^{-1}\Lambda).$ By Lemma \ref{tensorrat}, $g_k(\fa^{-1}\Lambda)\in\Qbar$ so $\alpha\in\Qbar$ if, and only if, $g_k(\Lambda')\in\Qbar$.

    \item We define $\alpha_\sigma$ the analytic representation of $(f^\sigma,\phi_{\Lambda_\sigma},\phi_{\Lambda'_\sigma}).$ We want to show that $\alpha_\sigma=\alpha^\sigma.$ Since $\Lambda_\sigma$ and $\Lambda'_\sigma$ are such that $g_i(\Lambda_\sigma)\in\overline{\QQ}$ and $g_i(\Lambda'_\sigma)\in\overline{\QQ}$, by $1.$ $\alpha_\sigma\in\overline{\QQ}$. According to \cite[Theorem~2.2]{Sil}, 
    \begin{equation}\label{eqnsigma}
            \ker( f^\sigma)=(\ker f)^\sigma=\cap_{a\in\fa} (\ker[a]_E)^\sigma=\cap_{a\in\fa} \ker[a^\sigma]_{E^\sigma}=E^\sigma[\fa^\sigma].
    \end{equation}

    Since $f^\sigma=\left({\nu_{\alpha}}_{|(\fa\star E)(\overline{\QQ})}\right)^\sigma\circ p^\sigma$ and $\left({\nu_{\alpha}}_{|(\fa\star E)(\overline{\QQ})}\right)^\sigma$ is an isomorphism, $\ker f^\sigma=\ker p^\sigma$. By definition of $E[\fa]$ we have \begin{equation}\label{eqnkernel}
        [\wp(z,\Lambda)\colon\wp'(z,\Lambda)\colon 1]\in E[\fa] \text{ if, and only if, }z\in\fa^{-1}\Lambda.
    \end{equation} Moreover, $[\wp(z,\Lambda_\sigma)\colon\wp(z,\Lambda_\sigma)\colon 1]\in\ker p^\sigma$ if, and only if, $z\in(\fa^{-1}\Lambda)_\sigma$ by definition of $p^\sigma$ but, $z\in {(\fa^\sigma)}^{-1}\Lambda_\sigma$ because $\ker p^\sigma=E^\sigma[\fa^\sigma]$. Thus, by the relations (\ref{eqnsigma}) and (\ref{eqnkernel}), we have $(\fa^{-1}\Lambda)_\sigma={(\fa^\sigma)}^{-1}\Lambda_\sigma$ (this proves the point 3 of the proposition).

    This proves that $(\fa\star E)^\sigma=\fa^\sigma\star E^\sigma.$ We also have the factorization $$f^\sigma=\left({\nu_{\alpha}}_{|(\fa\star E)(\overline{\QQ})}\right)^\sigma\circ p^\sigma={\nu_{\alpha_\sigma}}_{|(\fa^\sigma\star E^\sigma)(\overline{\QQ})}\circ p^\sigma.$$
    Since $p^\sigma$ is an isogeny it is surjective. So the maps $({\nu_\alpha}_{|\PP^2(\overline{\QQ})})^\sigma$ and ${\nu_{\alpha_\sigma}}_{|\PP^2(\overline{\QQ})}$ coincide and then $$\frac{1}{\alpha_\sigma^2}=\frac{1}{\left({\alpha^\sigma}\right)^2}\text{ so } \alpha_\sigma=\pm\alpha^\sigma\text{ and }\frac{1}{\alpha_\sigma^3}=\frac{1}{\left({\alpha^\sigma}\right)^3}\text{ so } \alpha_\sigma=j\alpha^\sigma$$
    for some $j^3=1.$ Hence $\alpha_\sigma=\alpha^\sigma.$
    
\end{enumerate}

\end{proof}

A nice consequence is that if we take $E_\Lambda$ with $g_2(\Lambda)\in\overline{\QQ}$ and $g_3(\Lambda)\in\overline{\QQ}$ then $E_{\fa^{-1}\Lambda}$ has also its Weierstrass model over $\overline{\QQ}$ and $\Hom_\CC(E_\Lambda,E_{\fa^{-1}\Lambda})=\Hom_{\Qbar}(E_\Lambda,E_{\fa^{-1}\Lambda})$ so the analytic representations of these isogenies associated with the isomorphisms $\phi_\Lambda$ and $\phi_{\fa^{-1}\Lambda}$ are in $\Hom_\CC(\Lambda,\fa^{-1}\Lambda)\subset K.$ We recall from \cite[Proposition 1.2]{Sil} that $\Cl(R)$ acts simply transitively on the elliptic curves with CM by $R$. Thus, given two elliptic curves $E$ and $E'$ over $\Qbar$ we can always find a model $E''$ of $E'$ such that the analytic representations of isogenies from $E$ to $E''$ are in $K.$

Given an automorphism $\sigma\in\GalQK$ and $F(\sigma)=\dot{\fa},$ with $F$ defined in (\ref{eq1}), by \cite[Proposition 2.4]{Sil} we know that $E^\sigma$ and $\fa\star E$ are isomorphic. This means that their corresponding lattices $\Lambda_\sigma$ and $\fa^{-1}\Lambda$ (via the $\wp$-functions) are homothetic by some constant $r_\sigma\in\CC$ which depends, a priori, on $\Lambda,\sigma$ and the choice of the representative $\fa$ of $F(\sigma)$ we chose. An immediate consequence of Proposition \ref{lemration} is that $r_\sigma$ is in $\Qbar$. The issue is that choosing another lattice $\Lambda'$ could lead to another $r_\sigma'\in\Qbar$ such that $r_\sigma'\Lambda_\sigma'=\fa^{-1}\Lambda'$. The next proposition shows that, under some condition, we can choose the same $r_\sigma$ for different lattices.

\begin{prop}[Action of $\Gal(\overline{\QQ}/K)$ on elliptic curves]\label{lemQK}
Let $\sigma\in\Gal(\overline{\QQ}/K)$ be an automorphism and $\fa$ be a fractional ideal with $F(\sigma)=\dot{\fa}\in\Cl(R)$. Consider an elliptic curve $E_\Lambda$ over $\Qbar$. Then there exists $r_\sigma\in\Qbar$ such that for all $E'(\CC)\underset{\phi_{\Lambda'}}{\simeq} \CC/\Lambda'$ such that $\Hom_\CC(\Lambda,\Lambda')\subset K$ and all isogenies $f\colon E_\Lambda\rightarrow E_{\Lambda'}$ over $\Qbar$ we have a commutative diagram
\[
\xymatrix{
E^\sigma(\CC)\ar[r]^{f^\sigma} & E'^\sigma(\CC)\\
\CC/\Lambda_\sigma  \ar[u]^{\phi_{\Lambda_\sigma}}\ar[r]^{\alpha}\ar[d]_{r_\sigma}&\CC/\Lambda'_\sigma \ar[u]_{\phi_{\Lambda'_\sigma}} \ar^{r_\sigma}[d]\\
\CC/\fa^{-1}\Lambda \ar[r]^{\alpha} & \CC/\fa^{-1}\Lambda'.}
\]
where $\alpha\in K$ is the analytic representation of $(f,\phi_{\Lambda},\phi_{\Lambda'})$.
\end{prop}
\begin{proof} Since the group of fractional ideals acts transitively on CM elliptic curves over $\Qbar$ there exists an isomorphism $E^\sigma\rightarrow \fa\star E$ which induces an isomorphism on the associated tori $r_\sigma\colon\CC/\Lambda_\sigma\rightarrow \CC/\fa^{-1}\Lambda.$ 
Let $\alpha\in K$ be the analytic representation of $(f,\phi_{\Lambda},\phi_{\Lambda'})$. There exists an ideal $\fb\subset R$ such that $\alpha\Lambda=\fb^{-1}\Lambda'$ and Proposition \ref{lemration}.3 implies that $\alpha^\sigma\Lambda_\sigma=(\fb^{-1})^\sigma\Lambda'_\sigma$. Since $\alpha\in K$ and $\fb\subset K$ they are invariant by $\sigma.$ To prove the proposition, we need to show that $r_\sigma\Lambda_\sigma'=\fa^{-1}\Lambda'$.\\
On one hand, we have
\begin{align}
    \alpha\Lambda_\sigma&=\fb^{-1}\Lambda'_\sigma,\label{EQ1}\\
\intertext{on the other hand}
    \alpha\fa^{-1}\Lambda&=\fa^{-1}\fb^{-1}\Lambda'.\label{EQ2}\\
\intertext{Combined with $r_\sigma\Lambda_\sigma=\fa^{-1}\Lambda$, (\ref{EQ1}) and (\ref{EQ2}), it gives}
    r_\sigma\Lambda'_\sigma&=\fb\alpha\fa^{-1}\Lambda\nonumber\\
        &=\fb\fa^{-1}\fb^{-1}\Lambda'\nonumber\\
        &=\fa^{-1}\Lambda'\nonumber
\end{align}
which concludes the proof.
\end{proof}

\subsubsection{Galois action on polarizations of CM elliptic curves}
The canonical polarization on an elliptic curve $E$ is defined by 
$$\begin{array}{rrl}
    a_{0,E}\colon & E & \longrightarrow \widehat{E}  \\
     & P & \longmapsto [P]-[O]
\end{array}$$
with $O$ the neutral element of the group $E$. We use this isomorphism to identify any polarization $a$ of $E$ with an element of $\End(E)$ by $a_{0,E}^{-1}\circ a$.
\begin{lem}\label{polelli}
Let $\fa\subset K$ be a fractional $R$-ideal and let $x\in\CC.$ Let $E(\CC)\underset{\phi_{\Lambda}}{\simeq}\CC/\Lambda$ be the elliptic curve associated with the lattice $\Lambda=\fa x$. The principal polarization $a_{0,E}$ induces the hermitian form 
$$\begin{array}{rrl}
    h_{0,\Lambda}\colon & \CC\times\CC & \longrightarrow \CC  \\
     & (z,w) & \longmapsto \frac{z\overline{w}}{\alpha_R N(\fa)N(x)}
\end{array}$$
with $\alpha_R=\im\omega>0$ and $R=\ZZ[\omega].$
\end{lem}
\begin{proof}
The induced hermitian form is necessarily of the form $h_{0,\Lambda}=\mu h_0$ with $\mu\in\RR_{>0}$ and $h_0(z,w)=z\overline{w}$ because the conjugacy classes of hermitian forms on $\CC$-vector spaces are determined by their rank and signature. Moreover, since $h_{0,\Lambda}$ is a principal polarization $\im h_{0,\Lambda}(\Lambda,\Lambda)=\deg a_{0,E}\ZZ=\ZZ$. Hence
\begin{align*}
    \im h_{0,\Lambda}(\Lambda,\Lambda)&= \im h_{0,\Lambda}(\fa x,\fa x)\\
    &=\im \mu N(\fa)N(x)R\\
    &=\mu N(\fa)N(x)\im R\\
    &=\mu N(\fa)N(x)\alpha_R\ZZ=\ZZ.
\end{align*}
Hence $\mu=\frac{1}{N(\fa)N(x)\alpha_R}.$
\end{proof}
 We want to investigate first how $\GalQK$ acts on polarizations.

Let $\phi_\Lambda\colon \CC/\Lambda\rightarrow E(\CC)$ with $\Lambda=\fb x.$ We write $\Rh(E_\Lambda,a_{0,E})= (\Lambda,H_{0,\Lambda})$ with 
$H_{0,\Lambda}=h^{\alpha_R}_{0,\Lambda}$ and, by Lemma \ref{polelli}, we can write the Gram matrix of $H_{0,\Lambda}$ in the $K$ basis $x$ of $K\Lambda$, by $G_{\Lambda}(x)=\frac{1}{N(\fb)}$.
It is then clear that for all fractional ideal $\fa$, we have $\fa\Lambda=\fa\fb x$ and thus, the Gram matrix of the analytic representation of the canonical polarization of $E_{\fa\Lambda}$ in the basis $x$ is $G_{\fa\Lambda}(x)=\frac{1}{\fa\fb}=\frac{1}{N(\fa)}G_\Lambda(x)$.

Moreover, $a_{0,E}^\sigma$ is the canonical polarization $a_{0,E^\sigma}$ of $E^\sigma$ for any $E$ and $\sigma\in\Aut(\CC)$ and every polarization on elliptic curves is of the form $n a_{0,E}$ for some $n\in\NN^*$. What we did can easily be generalized for any polarization of elliptic curves.

We can conclude with this proposition.
\begin{prop}\label{dim1}
Let $E=E_\Lambda\in\AVR,a_{0,E}$ the canonical polarization on $E$ and $\Rh(E,a_{0,E})=(\Lambda,H_{0,\Lambda})$. Then for any $\sigma\in\Gal(\overline{\QQ}/K)$ and $\dot{\fa}^{-1}=F(\sigma)$ there is an isometry
$$\Rh(E^\sigma,a_{0,E}){\simeq} \left(\fa\Lambda,\frac{1}{N(\fa)} H_{0,\Lambda}\right).$$
\end{prop}
\begin{proof} Let $\Rh(E^\sigma,a_{0,E^\sigma})=(\Lambda_\sigma,H_{0,\Lambda_\sigma}).$ By Proposition \ref{lemQK} and Lemma \ref{nur} we have an isomorphism $$\nu_{r_\sigma}\colon E^\sigma\longrightarrow  \fa^{-1}\star E$$
and $(\fa^{-1}\star E)(\CC)\simeq \CC/\fa\Lambda$. We have $\Rh(\fa^{-1}\star E,a_{0,\fa^{-1}E})= \left(\fa\Lambda,\frac{1}{N(\fa)} H_{0,\Lambda}\right).$ Since every isogeny between elliptic curves is a polarized isogeny for the canonical polarizations, the isomorphism $\nu_{r_\sigma}$ is a polarized isogeny and then its analytic representation $r_\sigma\colon(\Lambda_\sigma,H_{0,\Lambda_\sigma})\longrightarrow\left(\fa\Lambda,\frac{1}{N(\fa)} H_{0,\Lambda}\right)$ defines an isometry on the induced hermitian lattices.
\end{proof}
 
\subsection{Product of CM elliptic curves}\label{sectprod}
Let $A=\bigoplus_{i=1}^g E_i$ be the product of $g$ elliptic curves with CM by $R$ over $\Qbar$. Let $\Lambda_i\subset \CC$ be lattices such that $\phi_{\Lambda_i}\colon \CC/\Lambda_i\rightarrow E_i(\CC)$ is the canonical isomorphism with $g_k(\Lambda_i)\in\Qbar$. Then there is a canonical isomorphism $\phi_\Gamma\colon\CC^g/\Gamma\rightarrow A(\CC)$ with $\Gamma=\bigoplus \Lambda_i$ and $\phi_\Gamma=(\phi_{\Lambda_i})_{i=1\dots g}.$ In this section we show how the results of Section \ref{sectelliptic} apply to products of elliptic curves and we conclude with the proofs of Theorem \ref{KQ} and \ref{QK}.

\subsubsection{Isogenies between products of elliptic curves}

\begin{prop}\label{matrixgal}
Let $A,A'\in\AVR$ considered over $\Qbar$ with $A= E_{\Lambda_1}\times\dots\times E_{\Lambda_g}$ and $A'= E_{\Lambda'_1}\times\dots\times E_{\Lambda_{g}'}$. Let $\Gamma=\bigoplus_i\Lambda_i$ and $\Gamma'=\bigoplus_j\Lambda'_j$. Then, for any morphism $f\colon A\rightarrow A',$ the matrix $M_f$ of the analytic representation of $(f,\phi_\Gamma,\phi_{\Gamma'})$ in the canonical basis of $\CC\Gamma$ and $\CC\Gamma'$ has coefficients in $\overline{\QQ}$ and for any $\sigma\in\Gal(\overline{\QQ}/\QQ), M_{f^\sigma}=M_f^\sigma.$
\end{prop}
\begin{proof}
Consider the following diagram for all $k,l$
\[
\xymatrix{
\CC/\Lambda_k \ar@{^{(}->}[r]^{j_k}\ar[d]^{\phi_{\Lambda_k}} & \CC^g/\Gamma \ar[r]^{M_f}\ar[d]^{\phi_{\Gamma}}& \CC^g/\Gamma' \ar@{->>}[r]^{p'_l}\ar[d]^{\phi_{\Gamma'}} & \CC/\Lambda'_l\ar[d]^{\phi_{\Lambda_l'}}\\
E_{\Lambda_k}(\CC) \ar@{^{(}->}[r]^{\iota_k} & A(\CC) \ar[r]^{f} & A'(\CC) \ar@{->>}[r]^{\pi'_l} & E_{\Lambda'_l}(\CC)
}
\]
with $j_k,\iota_k$ and $p'_l,\pi_l'$ the $k$-th component inclusion and $l$-th component projection respectively.

Through those morphisms, we can see the $k,l$ coefficient $m_{l,k}$ of $M_f$ as a map
$$m_{l,k}\colon \CC/\Lambda_k\rightarrow \CC/\Lambda'_l$$ which is in $\overline{\QQ}$ and, by Proposition \ref{lemration}, $m_{l,k}^\sigma$ is the analytic representation of $((p_l'\circ f\circ\iota_k)^\sigma,\phi_{\Lambda_k},\phi_{\Lambda_l})$.
\end{proof}
Let $(A,a)\in \PAV$ with $A\simeq E_{\Lambda_1}\times\dots\times E_{\Lambda_g}$. Let $\Gamma=\bigoplus_i\Lambda_i$. We denote by $h$ the hermitian form on $\CC^g$ and $\rho_h\in\Hom_\CC(\Gamma,\widehat{\Gamma})$ the map induced by the polarization $a$. Consider $a_{0}$ the product polarization of the canonical polarizations on each elliptic curve $a_{0,E_i}\colon E_i\rightarrow \widehat{E_i}$. It is an isomorphism
$$a_{0}\colon \bigoplus_{i=1}^g E_i\rightarrow \bigoplus_{i=1}^g \widehat{E_i}$$
which induces
$$\rho_{h_0}=( \rho_{h_{0,\Lambda_i}})\colon \CC^g/\bigoplus_i \Lambda_i\rightarrow (\CC^g)^*/\bigoplus_i\widehat{\Lambda_i} $$
with each $\rho_{h_{0,\Lambda_i}}$ induced by the canonical polarization on $E_{\Lambda_i}$.
We can now consider the analytic representation $\rho$ of $(a_{0}^{-1}\circ a,\phi_\Gamma,\phi_\Gamma)$. By definition, the diagram of Figure \ref{dual} commutes.
\begin{center}
   \begin{figure}[ht]
$$\begin{tikzcd}
 &\text{\footnotesize $ \CC^g/\Gamma$} \ar[rd, "\rho_{h_0}"]\ar[dd,"\text{\tiny $\phi_\Gamma$}"near start] & \\
\CC^g/\Gamma \ar[rr,"\rho_h" near start, crossing over]\ar[dd,"\phi_{\Gamma}"']\ar[ru,"\rho"]&   &(\CC^g)^*/\widehat{\Gamma}  \ar[dd] \\
    & \text{\footnotesize $A(\CC)$} \ar[rd,"a_0"]&   \\
A(\CC)\ar[ru,"a_0^{-1}\circ a" near end] \ar[rr,"a"] &   & \widehat{A}(\CC) &
\end{tikzcd}$$
\caption{Analytic representation of polarizations}\label{dual}
\end{figure} 
\end{center}
Recall that we denoted $\alpha_R=\im\omega>0$ with $R=\ZZ[\omega].$ Let $\Lambda_i=\fa_i x_i$ then $\Gamma=\bigoplus \fa_i x_i$ and, with Lemma \ref{integ} we can show that $$\widehat{\Gamma}=\left\{\ell \in (\CC^g)^*\mid \im \ell(\Gamma)\subset \ZZ\right\}=\left\{\ell \in (\CC^g)^*\mid \alpha_R\ell(\Gamma)\subset R\right\}=\bigoplus \frac{1}{\alpha_R}\overline{\fa_j}^{-1}x_j^*$$
with $x_j^*\in(\CC^g)^*$ defined by $x_j^*(x_i)=\delta_{i,j}$ with $\delta_{i,j}=1$ for $i=j$ and $0$ otherwise.
\begin{prop}\label{gramherm}
With the notation above let $b=(x_i)_{i=1,\dots,g}$. Let $M_{b,b}(\rho)$ be the matrix of $\rho$ in the basis $b$ and $G_{\Gamma}(b)=(h^{\alpha_R}(x_i,x_j))_{i,j}$ the Gram matrix of the hermitian form $h^{\alpha_R}=\alpha_R h$ in the basis $b$. Then $$M_{b,b}(\rho)= G_{\Gamma}(b)\cdot D$$
with $D=\diag(N(\fa_1),\dots,N(\fa_g))$, the diagonal matrix with coefficients $d_{i,i}=N(\fa_i).$
\end{prop}
\begin{proof}
By definition of $\rho_{h_{0,\Lambda_i}}\in \Hom_\CC(\Lambda_i,\widehat{\Lambda_i})$ and Lemma \ref{polelli}
$$\begin{array}{rrl}
 \rho_{h_{0,\Lambda_i}}\colon & \CC/\Lambda_i & \longrightarrow  \CC^*/\widehat{\Lambda_i}  \\
     & z &\longmapsto \left(\ell_z\colon w\mapsto\frac{z\overline{w}}{\alpha_R N(\fa)N(x_i)}\right).
\end{array}$$
Hence $\rho_{h_0}(x_i)=\ell_{x_i}=\frac{1}{N(\fa_i)\alpha_R}x_i^*$ so $\rho_{h_0}^{-1}(x_i^*)=N(\fa_i)\alpha_R x_i$. Thus, for any $\ell=\sum_i u_i x_i^*\in(\CC^g)^*,$
\begin{equation}
     \rho_{h_0}^{-1}(\ell)=\sum_{i=1}^g u_i  N(\fa_i)\alpha_R x_i \label{eq2}
\end{equation} and then its $j$-th component is
\begin{equation}
    (\rho_{h_0}^{-1}(\ell))_j= u_j N(\fa_j)\alpha_R = \alpha_R \ell(x_j)\label{eq3}
\end{equation}
Now, since $\rho=\rho_{h_0}^{-1}\circ \rho_h$ we have
\begin{align*}
    \rho_{ij}&=(\rho(x_i))_j\\
    &=(\rho_{h_0}^{-1}\circ \rho_h(x_i))_j\\
    &=N(\fa_j)\alpha_R \rho_h(x_i)(x_j) \text{ by (\ref{eq2}) and (\ref{eq3}) }\\
    &=N(\fa_j)\alpha_R h(x_i,x_j)\\
    &=N(\fa_j)h^{\alpha_R}(x_i,x_j)\\
    &=(G_\Gamma(b)\cdot D)_{i,j}.
\end{align*}
\end{proof}

\begin{prop}\label{propQK}
Let $A,A'\in\AVR$ considered over $\Qbar$ with $A = E_{\Lambda_1}\times\dots\times E_{\Lambda_g}$ and $A'= E_{\Lambda'_1}\times\dots\times E_{\Lambda_{g'}}$. Let $\Gamma=\bigoplus_i\Lambda_i$ and $\Gamma'=\bigoplus_j\Lambda'_j$. Assume that for all $i,j, \Hom_\CC(\Lambda_i,\Lambda'_j)\subset K.$ Then for any $\sigma\in\Gal(\overline{\QQ}/K)$ and $\fa$ a fractional $R$-ideal such that $F(\sigma)=\dot{\fa}\in\Cl(R)$ there exists $r_\sigma\in\Qbar$ such that for any $f\colon A\rightarrow A'$ with analytic representation $M$ there is a commutative diagram
\[
\xymatrix{
A^\sigma(\CC)\ar[r]^{f^\sigma} & A'^\sigma(\CC)\\
\CC^g/\Gamma_\sigma  \ar[u]^{\phi_{\Gamma}}\ar[r]^{M}\ar[d]_{r_\sigma I_{g}}&\CC^{g'}/\Gamma'_\sigma \ar[u]_{\phi_{\Gamma'_\sigma}} \ar^{r_\sigma I_{g'}}[d]\\
\CC^g/\fa^{-1}\Gamma \ar[r]^{M} & \CC/\fa^{-1}\Gamma'.}
\]
\end{prop}
\begin{proof}
Apply Lemma \ref{lemQK} and following the same steps as in the proof of Proposition \ref{matrixgal}.
\end{proof}
\begin{prop}\label{lemlibre}
Let $E_1,\dots,E_g$ be elliptic curves over $\overline{\QQ}$ with CM by $R$ and the isomorphisms $\phi_{\Lambda_i}\colon \CC/\Lambda_i\rightarrow E_i(\CC)$ with $\Hom_\CC(\Lambda_i,\Lambda_j)\subset K$. Let $a_{0,E_i}$ be the canonical polarization on $E_i$, let $a_0$ be the product polarization on $E_1\times\dots\times E_g$ and $\Rh(\bigoplus_i E_i,a_0)=(\bigoplus_i\Lambda_i,H_0)$. Then for any $\sigma\in\Gal(\overline{\QQ}/K)$ and $\dot{\fa}^{-1}=F(\sigma)$ there is an isometry
$$\Rh\left(\bigoplus_i E_i^\sigma,a_0^\sigma\right){\simeq} \left(\fa\bigoplus_i\Lambda_i,\frac{1}{N(\fa)}H_0\right).$$
\end{prop}
\begin{proof}
Apply Lemma \ref{dim1} component by component.
\end{proof}

\subsubsection{Proof of Theorem \ref{KQ} and \ref{QK}}
We have now the necessary tools to prove both Theorem \ref{KQ} and \ref{QK}. Since both theorems aim to describe the isometry class of $\Rh(A^\sigma,a^\sigma)$, for $(A,a)\in\PAV$ and $\sigma\in \GalQK$ or $\sigma\in \GalKQ$, it is enough to show that there is such an isometry for a particular object in the isomorphism class of $(A^\sigma,a^\sigma).$ By definition of $\PAV$, each isomorphism class of polarized abelian variety contains an element of the form $(E_1\times \dots\times E_g,a)$ so we will show the isometry for this particular case.

Let $(L,H)$ be an integral rank $g$ hermitian $R$-lattice. Fixing a basis $b$ of $KL$ gives an isomorphism $KL\simeq K^g$ and pushing forward $H$ on $K^g$ gives an isometry $(KL,H)\rightarrow (K^g,H)$. We identify $(L,H)$ with its image in $(K^g,H).$ The hermitian form $H$ is determined by the Gram matrix of $b$ given by $G=G(b)=(H(v,w))_{v,w\in b}\in M_{g,g}(K).$ Indeed, with $x,y\in K^g$, we have $$H(x,y)=\trsp x G \overline{y}$$ and then, by definition,
$$\overline{H}(x,y)=\overline{H(\overline{x},\overline{y})}=\overline{\trsp\overline{x}G \overline{\overline{y}}}=\trsp x \overline{G} \overline{y}.$$
Hence $(\overline{L},\overline{H})$ has Gram matrix $\overline{G}$.
\begin{proof}[Proof of Theorem~\ref{KQ}] Let $(A,a)\in\PAV$ considered over $\Qbar$ with $A= E_{\Lambda_1}\times\dots\times E_{\Lambda_g}$ with $\Hom_\CC(\Lambda_i,\Lambda_j)\subset K$. Let $\Rh(A,a)=(L,H)$ and let $\Rh(\overline{A},\overline{a})=(L',H')$. We write $\Lambda_i=\fa_i x_i$ and $b=(x_1,\dots,x_g)$ such that $(\fa_i,x_i)$ is a pseudo-basis of $L=\bigoplus_i\Lambda_i$. By Lemma \ref{compconj}, $L'=\overline{L}$. Moreover, by Proposition \ref{gramherm}, the matrix $M(b,b)$ of the analytic representation $\rho$ of $(a_0^{-1}\circ a,\phi_L,\phi_L)$ satisfies $M_{b,b}(\rho)=G_L(b)\cdot D$ with $D$ the diagonal matrix $\diag\parent{N(\fa_1),\dots,N(\fa_g)}$. By Proposition \ref{matrixgal} and by applying the complex conjugation to the diagram

$$\xymatrix{\CC^g/L\ar[d]^{\phi_L}\ar[r]^{\rho} & \CC^g/L\ar[d]^{\phi_L}\\
A(\CC)\ar[r]^{a_0^{-1}\circ a}& A(\CC)} $$
the analytic representation of $(\overline{a_0}^{-1}\circ\overline{a},\phi_{\overline{L}},\phi_{\overline{L}})$ is $\overline{M_{b,b}(\rho)}=\overline{G_{L}(b)}$. Hence $(\overline{L},H')$ has Gram matrix $\overline{G_L(b)}$ so $H'=\overline{H}.$
\end{proof}

\begin{proof}[Proof of Theorem~\ref{QK}]
Let $(A,a)\in\PAV$ considered over $\Qbar$ such that $A=E_{\Lambda_1}\times\dots\times E_{\Lambda_g}$ with $\End(\Lambda_i,\Lambda_j)\subset K$ for all $i,j$. We will write $E_i=E_{\Lambda_i}$ for simplicity. Let $\Gamma=\bigoplus \Lambda_i$ and $\Rh(A,a)= (L,H), \sigma\in\Gal(\overline{\QQ}/K)$ with $L=\Gamma$ and $\dot{\fa}^{-1}=F(\sigma)$.
\begin{description}
\item[First step : find a polarized isogeny $(E^n, D\cdot\lambda_0)\rightarrow (A,a) $ with $D$ diagonal with integer entries.] Let $\Lambda=\Lambda_1$ and $E=E_{1}$ and $\lambda_0$ the product polarization on $E^g.$ There exist $\alpha_i\in K$ such that for all $i,\alpha_i\Lambda\subset \Lambda_i$. This induces an isogeny $q\colon E^g\rightarrow A$. We denote by $Q=\diag(\alpha_1,\dots,\alpha_g)$ the analytic representation of $(q,\phi_{\Lambda^g},\phi_{\Gamma})$. We consider $\lambda\colon E^g\rightarrow \widehat{E}^g$ the pullback polarization of $a$ on $E^g$, i.e., the unique polarization on $E^g$ that makes $q$ a polarized isogeny. We consider $M=\lambda_0^{-1}\circ\lambda\in\End(E^g)\simeq M_g(R)$, and since $\lambda=\widehat{\lambda}$ the matrix $M$ is a hermitian matrix, i.e., $\trsp\overline{M}=M.$ Now consider a matrix $P\in M_g(R)$ such that $\trsp\overline{P}M P = D$ with $D$ a diagonal matrix. Since $M$ is positive definite and hermitian, so does $D$. Since $R\cap \RR^+=\NN$ the matrix $D$ has integer entries that are positive. So we have a polarized isogeny $f=q\circ P\colon (E^g,D\cdot\lambda_0)\rightarrow (A,a)$.
\[
\xymatrix{
E^g \ar[r]^{P} \ar[d]^{D}& E^g \ar[r]^q \ar[d]^M & \bigoplus E_i \ar[d]^{a}\\
E^g  & E^g \ar[l]^{\trsp \overline{P}} & \ar[l]^{\widehat{q}}\widehat{\bigoplus E_i}\simeq\bigoplus E_i
}
\]

We will denote by $S=QP$ the analytic representation of $(f,\phi_{\Lambda^g},\phi_\Gamma).$ Since $f$ is a polarized isogeny
\begin{equation}\label{eqS}
    S\colon (K\Lambda^g,D)\rightarrow (KL,H)
\end{equation}
is an isometry.

\item[Second step : conclude.] By Proposition \ref{lemlibre}, $\Rh\left((E^g)^\sigma,D^\sigma\right)\underset{r_\sigma}{\simeq} \left(\fa \Lambda^g,\frac{1}{N(\fa)}D\right).$ We now consider $\Rh(A^\sigma,a^\sigma)=(L_\sigma,H_\sigma)$ and $\iota_\sigma=\Rh(f^\sigma)$. By Proposition \ref{lemQK} we have $r_\sigma L_\sigma=\fa L$ and there is a commutative diagram
\[
\xymatrix{
(E^g)^\sigma(\CC)\ar[r]^{f^\sigma} & A^\sigma(\CC) \\
 \CC^{g}/\Lambda_\sigma^g \ar[u]^{\phi_{P_\sigma}}\ar[r]^{S}\ar[d]_{r_\sigma I_{g}}&\CC^g/L_\sigma \ar[u]_{\phi_{L_\sigma}} \ar^{r_\sigma I_{g}}[d]\\
\CC/\fa \Lambda^g \ar[r]^{S} & \CC^g/\fa L.}
\]
Moreover, the maps
\begin{equation*}
    r_\sigma I_g\colon (KL_\sigma,H_\sigma)\rightarrow \left(K\fa L,H'=\frac{1}{N(r_\sigma)}H_\sigma\right)\text{ and }S\colon \left(K\fa \Lambda^g,\frac{1}{N(\fa)}D\right)\rightarrow \left(K\fa L,H'\right)
\end{equation*}
are isometries. Hence we have $H'=\frac{1}{N(\fa)}\trsp\overline{S^{-1}}DS^{-1}=\frac{1}{N(\fa)}H$.
Hence $$\Rh(A^\sigma,a^\sigma)= (L_\sigma,H_\sigma)\simeq \left(\fa L,\frac{1}{N(\fa)}H\right).$$
This concludes the proof of Theorem \ref{QK}.
\end{description}
\end{proof}

\section{Application to the field of moduli of varieties in $\PAV$}\label{sect4}

\subsection{General results on the field of moduli of principally polarized abelian varieties in $\PAV$}\label{VAPPandCurves}
We recall that the \textit{field of moduli} of a polarized abelian variety $(A,a)$ over $\Qbar$ is the field fixed by the subgroup
$$\left\{\sigma\in\GalQQ\mid (A^\sigma,a^\sigma)\simeq (A,a)\right\}.$$
In the same way we define the field of moduli of a curve $C$ over $\overline{\QQ}$ by the fixed field of $\left\{\sigma\in\GalQQ\mid C^\sigma\simeq C\right\}.$

The Abel-Jacobi map $C\mapsto (\Jac(C),j)$ which associates a curves to its polarized Jacobian variety induces a morphism
$$\begin{array}{rrl}
    [\Jac]\colon & M_g \longrightarrow & A_g \\
     & [C] \longmapsto & [\Jac(C),j] 
\end{array}$$
between the moduli space of smooth absolutely irreducible projective genus $g$ curves and the moduli space of principally polarized abelian varieties of dimension $g$. By \cite[Chapter~VII,Corollary~12.2]{SilCor}, $[\Jac]$ is injective so $C$ and $(\Jac(C),j)$ have the same field of moduli.

For $g=2$ and $g=3$ the moduli spaces $M_g$ and $A_g$ have the same dimension and are irreducible. Thus, the Jacobian map is dominant if the field is algebraically closed. More specifically, every indecomposable principally polarized abelian variety is the Jacobian of a curve (see \cite{Hoy}).

We give a necessary condition on the class group of $R$ and on $\Rh(A,a)$ for an abelian variety $(A,a)\in\PAV$ to have field of moduli $\QQ.$
\begin{prop}\label{exponent}
Let $(A,a)\in\PAV$ considered over $\Qbar$ and $\Rh(A,a)=(L,H)$ be a hermitian lattice of rank $g$. If $(A,a)$ has field of moduli $\QQ$ then $\Cl(R)$ has exponent dividing $g$ and $\st(L)$ has order at most $2.$
\end{prop}
\begin{proof}
If $(L,H)$ corresponds to a polarized abelian variety with field of moduli $\QQ$ then $(L,H)\simeq\left(\fa L,\frac{1}{N(\fa)}H\right)$ for all fractional ideal $\fa.$ In particular, $L\simeq \fa L$ for all $\fa$ and then, their Steinitz classes are the same
$$\st(L)=\st(\fa L)=\fa^g\st(L)\in\Cl(R).$$
Hence $\dot{\fa}^g=1\in\Cl(R)$ so $\Cl(R)$ has exponent dividing $g.$

The hermitian lattice isometry class must also be invariant by the action of the complex conjugation so $$\st(L)=\st(\overline{L})=\overline{\st(L)}.$$
By the formula $\overline{\fa}\fa=N(\fa)R$ it means that $\st(L)$ must have order at most $2$.
\end{proof}
\begin{coro}\label{corodd}
Let $(A,a)\in \PAV$ with odd dimension $g$. Suppose $(A,a)$ has field of moduli $\QQ$. Then there exists an elliptic curve $E$ with CM by $R$ such that $A\simeq E^g.$
\end{coro}
\begin{proof}
Let $(L,H)=\Rh(A,a)$ be the corresponding unimodular hermitian lattice. By Proposition \ref{exponent}, its Steinitz class $\st(L)$ has order dividing $g$ and $2$ hence it must be $1.$ In other words $L$ is free over $R$, i.e., $L\simeq R^g$ and then $A\simeq E^g$ with $\R(E)\simeq R.$
\end{proof}

\subsection{Enumeration of the indecomposable principally polarized abelian varieties in $\PAV$ with field of moduli $\QQ$}\label{sectalgo}
Since we are able to understand the action of $\GalQK$ and $\Gal(K/\QQ)$ through the equivalence of categories $\Rh$ developed in Section \ref{sect2} we are able to check when $(A,a)\simeq (A^\sigma,a^\sigma)$ for all $\sigma\in\GalQQ$ when $(A,a)\in\PAV$ by looking for isometries between hermitian $R$-lattices. This is what Algorithm \ref{algor} does.\\

\emph{We denote by $\AR(g)$ the set of all classes of dimension $g$ indecomposable principally polarized abelian varieties $(A,a)\in\PAV$ and by $\ARQ(g)$ the subset of $\AR(g)$ corresponding to elements of $\AR(g)$ with field of moduli $\QQ.$}

\begin{algorithm}[ht]

% enter the algorithm environment
\begin{algorithmic} \caption{Enumeration algorithm}\label{algor}% enter the algorithmic environment
	\REQUIRE An integer $g$ and a maximal order $R$ with exponent dividing $g$.
   \ENSURE The list of unimodular indecomposable hermitian lattices $(L,H)$ corresponding to the elements of $\ARQ(g)$.
\STATE $\LList\leftarrow\{ \text{Unimodular indecomposable hermitian lattices of rank }g\}/\simeq$
\STATE $\LListFM\leftarrow \{~\}$ \COMMENT{ List of abelian varieties with field of moduli $\QQ$.}
\FOR{$(L,H)\in\LList$}
   \STATE $\bool\leftarrow \TRUE$
\FOR{$\fa\in\{$generators of $\Cl(R)\}$}
   \STATE $\bool\leftarrow \bool$ and $(L,H)\simeq\left(\fa L,\frac{1}{N(\fa)}H\right)$.
\ENDFOR
   \IF{ $\bool$ and $(L,H){\simeq}              \left(\overline{L},\overline{H}\right)$} \STATE $\LListFM\leftarrow \LListFM\cup \{(L,H)\}$
   \ENDIF
\ENDFOR
\RETURN $\LListFM$

\end{algorithmic}

\end{algorithm}
To compute the list of elements of $\ARQ(g)$ for a given maximal order $R$ we need the list of all unimodular indecomposable hermitian $R$-lattices of rank $g$. This can be done using the classification of these lattices developed by authors in \cite{KNRR}.

We want to run the algorithm over all maximal orders of a given exponent dividing $g$. By \cite{EKN}, the complete list of the corresponding discriminants is finite for all $g$ and known for $g$ up to $8$ under the Extended Riemann Hypothesis. However, we can get rid of the Extended Riemann Hypothesis for the genus $2$ thanks to Proposition \ref{prop:g2} and for the genus $3$ thanks to Proposition \ref{prop:g3}, presented in the Appendix. Indeed, these two propositions provide an upper bound for the class number of a maximal order $R$ that an abelian variety isomorphic to a product of elliptic curves with complex multiplication by $R$ can have. This allows us to use the classification of imaginary quadratic fields with a given class number $h$, which is complete for $h\leq 100$ by \cite{Wat}.

\subsection{Enumeration of dimension $2$ and $3$ principally polarized abelian varieties of $\PAV$ with field of moduli $\QQ$}\label{g23}

As the computations become quickly time-consuming as the dimension $g$ and the discriminant of the order grow we restricted to $g=2$ and $3$ to be able to have complete tables. We used the Magma library developed in \cite{KNRR}. 

We use Algorithm 1 to compute the cardinality of $\ARQ(2)$ and present the results in the Table \ref{g2}. In the column $\mathcal{P}$ we copy the number of $(A,a)\in \PAV$ such that $A$ is the square of an elliptic curve with field of moduli $\QQ$ to confirm we find the same values as in \cite[Table~2]{GHR}.

\begin{table}[ht]
\centering
\begin{tabular}{
|c|@{\hbox to 0.1em{}}r@{\hbox to 0.7em{}}c@{\hbox to 0.1em{}}c@{\hbox to 0.1em{}}c||
c@{\hbox to 0.5em{}}|@{\hbox to 0.5em{}}r@{\hbox to 0.7em{}}c@{\hbox to 0.1em{}}c@{\hbox to 0.1em{}}c||
c|@{\hbox to 0.7em{}}r@{\hbox to 0.7em{}}c@{\hbox to 0.1em{}}c@{\hbox to 0.1em{}}c|}
\hline
$h_R$ & $\Delta$ & $\mathcal{P}$ & $\#\ARQ$ & $\#\AR$ &$h_R$ & $\Delta$ & $\mathcal{P}$ & $\#\ARQ$ &  $\#\AR$ & $h_R$ & $\Delta$ & $\mathcal{P}$ & $\#\ARQ$ & $\#\AR$\\ 
    \hline
1&-3 & 0 & 0 & 0&2&-15& 0& 1& 1&2&-115& 0 &3 &11\\
&-4 &0&0&0&&-20& 1& 1& 3&&-123& 0& 4& 12\\
&-7&0&0&0&&-24&1& 3& 3&&-148& 3 &5 &13\\
&-8 & 1 &1 &1&&-35& 0& 1 &5&&-187 &0 &3& 17\\
&-11 & 1& 1& 1&&-40 &2 &4 &4&&-232& 5& 10& 20\\
&-19 & 1 & 1& 1&&-51 &0 &2 &6&&-235& 0& 5& 21\\
&-43 &2 &2 &2&&-52& 2 &3& 5&&-267& 0 &6 &24\\
&-67 &3 &3 &3&&-88 &2 &6& 8&&-403& 0& 3& 35\\
&-163 &7 &7 &7&&-91 &0& 1 &9&&-427 &0 &3 &37\\
\hline \hline

$h_R$ & $\Delta$ & $\mathcal{P}$ & $\#\ARQ$ & $\#\AR$ &$h_R$ & $\Delta$ & $\mathcal{P}$ & $\#\ARQ$ &  $\#\AR$ & $h_R$ & $\Delta$ & $\mathcal{P}$ & $\#\ARQ$ & $\#\AR$\\ 
    \hline
4&-84 &0 &2 &18&4 &-340 &0&2& 60&4&-595 &2& 2 &106\\
&-120& 3& 4& 24&&-372&0 &2 &66&&-627& 0& 0& 112\\
&-132 &1 &2 &26&&-408 &0 &4 &72&&-708 &1 &2 &122\\
&-168& 0 &4& 32&& -435& 0&2& 80&&-715&0 &2 &126\\
&-195 &0 &2 &40&&-483 &0 &0 &88&&-760 &1 &4 &130\\
&-228&1&2& 42&&-520& 3&4& 90&&-795&2 &2 &140\\
&-280 &0&4 &50&&-532 &0 &2 &92&&-1012 &0 &2 &172\\
&-312 &1&4& 56&&-555&0&2& 100&&-1435&0 &2 &246\\
\hline
\end{tabular}
    \begin{tablenotes}
        \item $h_R\colon$ Class number of the maximal order $R$ of $\QQ(\sqrt{\Delta})$.
        \item$\#\AR\colon$ Number of elements of $\AR(2)$ defined in Section \ref{sectalgo}.
        \item$\#\ARQ\colon$ Number of elements of $\ARQ(2)$.
        \item$\mathcal{P} \colon$ Number of elements $(A,a)$ of $\ARQ(2)$ such that $A\simeq E^2$ for some $E$.
        \item
    \end{tablenotes}
\caption{Computations for $g=2$}\label{g2}
\end{table}

In fact, for $g=2$, the following proposition shows that it is not necessary to check if $(L,H)\simeq(\overline{L},\overline{H})$.
\begin{prop}\label{conjisgal}
Let $(A,a)$ be a dimension $2$ principally polarized abelian variety over $\CC$ with $A$ isomorphic to the power of elliptic curves $E_i$ with CM by $R$ maximal. Let $\Rh(A,a)=(L,H)$ be a hermitian integral lattice and let $\fa$ be the Steinitz class of $L$. Then $$\Rh\left(\overline{A},\overline{a}\right)=\left(\overline{L},\overline{H}\right)\simeq \left(\overline{\fa}L,\frac{1}{N(\fa)}H\right).$$
Hence in this particular case, the action of the complex conjugation corresponds to the action of an automorphism of $\Gal(\overline{\QQ}/K)$.
\end{prop}
\begin{proof}
Let us write $L=R x\oplus \fa y.$ Let $G= \begin{pmatrix} \alpha & \beta \\ \overline{\beta} & \delta \end{pmatrix}$ be the Gram matrix of $H$ in the basis $(x,y)$ of $KL$. Since $(A,a)$ is principally polarized $(L,H)$ must be unimodular and then its volume $\fv(L)=N(\fa)\det(G)R=R$ so $N(\fa)\det(G)$ is invertible and real in $R$ so $N(\fa)\det(G)=1.$ 

Now consider $P= N(\fa)\begin{pmatrix} \overline{\beta} & \delta \\ -\alpha & -\beta \end{pmatrix}$, matrix of a linear map $K\overline{x}+K\overline{y}\rightarrow K x+K y.$ It satisfies the relation $$\trsp P \frac{1}{N(\fa)}G \overline{P}= N(\fa)\det(G)\overline{G}=\overline{G}.$$ So $P$ defines an isometry between hermitian spaces.

Moreover, $\alpha=H(x,x)=N(x)\in R\cap \RR=\ZZ, \beta=H(x,y)$, so $\overline{\fa}\beta=H(x,\fa y)\subset R$ so $\beta\in \overline{\fa}^{-1}=\frac{\fa}{N(\fa)}$ and, in the same way, $\delta\in \frac{1}{N(\fa)}\ZZ.$
Hence 
\begin{align*}
    P\overline{x}&=N(\fa)(\overline{\beta}x-ay)\in\overline{\fa}x+N(\fa)y=\overline{\fa}L\\
    P\overline{\fa y}&=N(\fa)\overline{\fa}(\delta x-by)\in \overline{\fa}x+N(\fa)y=\overline{\fa}L.
\end{align*}
Thus, $P$ defines an isometry $(\overline{L},\overline{H})\rightarrow \left(\overline{\fa}L,\frac{1}{N(\fa)}H\right)$
\end{proof}

In dimension $g=3$ computations are more time consuming. Fortunately, by Corollary \ref{corodd}, in odd dimension all isomorphism class of $(A,a)$ in $\ARQ(3)$ are actually isomorphic to some a power of an elliptic curve. Hence we can run the algorithm only on free unimodular hermitian lattices and the latter are easier to enumerate and to work with.

We summarize the calculations in Table \ref{g3}. All computations use Algorithm \ref{algor}, except for the discriminant $-4027$, for which Proposition \ref{prop:g3} allows us to complete this missing entry.
\begin{table}[ht]
\centering
\begin{tabular}{|l|@{\hbox to 0.1em{}}r@{\hbox to 0.1em{}}c@{\hbox to 0.1em{}}r||l|@{\hbox to 0.1em{}}r@{\hbox to 0.1em{}}c@{\hbox to 0.1em{}}r|
}
\hline
$h_R$ & $\Delta$ & $\#\ARQ$ & $\#\AR^\text{free}$ & $h_R$ & $\Delta$ & $\#\ARQ$ & $\#\AR^\text{free}$ \\ 
    \hline
1&-3  & 0 & 0&3&-107  &2 &44\\
&-4 &0&0&&-139  &1 &79\\
&-7&0&0&&-211 &0& 209\\
&-8  &0 &0&&-283  & 1& 417\\
&-11&0& 0&&-307 &0 &507\\
&-19  &1& 1&&-331& 2 &613\\
&-43  &3 &5&&-379&0 &851\\
&-67  &5 &13&&-499& 1 &1665\\
&-163  &13 &103&&-547 & 1 &2059\\
 &     &  &  &&-643& 1 &3075\\
3&-23&0& 3&&-883& 0&6703\\
&-31&0& 6&&-907& 1 &7163\\
&-59&1& 10&& & &\\
&-83&0 &24 &9& -4027 & 0 & 0 \\
\hline
\end{tabular}
    \begin{tablenotes}
        \item $h_R\colon$ Class number of the maximal order $R$ of $\QQ(\sqrt{\Delta})$. 
        \item $\#\AR^\text{free}\colon$ Number of classes $(A,a)$ in $\AR(3)$ with $A\simeq E^3$ for some $E$.
        \item $\#\ARQ\colon$ Number of elements of $\ARQ(3)$.
        \item
    \end{tablenotes}
    
\caption{Computations for $g=3$}\label{g3}
\end{table}

\appendix
\section{Class groups of CM fields from polarized products of CM elliptic curves with field of moduli $\QQ$}\label{appA}
\begin{center}
    \textit{by Francesc Fité and Xavier Guitart}
\end{center}
Let $K$ be an imaginary quadratic field and let ${H_K}$ be its Hilbert class field. We denote by $\mathrm{C}_g$ the cyclic group of $g$ elements.
\begin{prop}\label{prop:g2}
Let $E_1/\Qbar$ and $E_2/\Qbar$ be elliptic curves with CM by $K$, and let $\varphi$ be an indecomposable principal polarization on $E_1\times E_2$. If the field of moduli of $(E_1\times E_2,\varphi)$ is $\QQ$, then $\Cl(K)$ is isomorphic to one of the groups $ \mathrm{C}_1,\mathrm{C}_2$, or $\mathrm{C}_2\times \mathrm{C}_2$.
\end{prop}
\begin{proof}
  Since $\varphi$ is indecomposable there exists a curve $C/\Qbar$ with field of moduli $\QQ$ such that $\Jac(C)\simeq E_1\times E_2$. The curve $C$ is not necessarily defined over $\QQ$. By \cite[\S2.4]{Mes91} there exists a number field $k'$ with $[k'\colon \QQ]\leq 2$ such that $C$ admits a model over $k'$. Since  $k'$ can be chosen to be any field over which Mestre's conic has a rational point, we can choose it so that $k'\cap {H_K}=\QQ$.

  Put $k = Kk'$, and suppose from now on that $C$ is defined over $k$.  In particular, $A=\Jac(C)$ is an abelian surface defined over $k$ such that $A_{\Qbar} \sim E^2$, where $E/\Qbar$ is an elliptic curve with CM by $K$. Denote by $L$ the smallest field of definition of the endomorphisms of $A$.  Since $k$ contains $K$, by \cite[Theorem 2.14]{FG} the field $F=Hk$ is a subfield of $L$ and $\Gal(F/k)$ has exponent $\leq 2$. By \cite[Remark 3.1]{FG} $\Gal(L/k)$ is either $\mathrm{C}_n$ for $n\in \{1,2,3,4,6\}$, $\mathrm{D}_n$ for $n\in \{2,3,4,6\}$, $\mathrm{A}_4$ or $\mathrm{S}_4$. Since $\Gal(F/k)$ is a quotient of exponent $\leq 2$ of $\Gal(L/k)$, we see that $\Gal(F/k)\simeq \mathrm{C}_1,\mathrm{C}_2$, or $\mathrm{C}_2\times \mathrm{C}_2$ (cf. \cite[Table1]{FG}). 
   By our choice of $k'$, we have that ${H_K}\cap k=K$, and hence $\Gal(F/k)\simeq \Gal({H_K}/K)\simeq \mathrm{Cl}(K)$.  
  
  %In particular $[F\colon k]\leq 4$, and therefore $[F\colon K]\leq 8$. Since ${H_K}\subset F$, we see that $[{H_K}\colon K]\leq 8$ and therefore $|\mathrm{Cl}(K)|\leq 8$. By \cite[Proposition 9]{Nar} we have that the exponent of $\mathrm{Cl}(K)$ is $\leq 2$ and this gives that $\mathrm{Cl}(K)\simeq  \mathrm{C}_1,\mathrm{C}_2,\mathrm{C}_2\times \mathrm{C}_2$, or $\mathrm{C}_2\times \mathrm{C}_2\times \mathrm{C}_2$.
\end{proof}
\begin{rem} The above proposition dispenses with the assumption of the Generalized Riemann Hypothesis in \cite[Table 4]{GHR}. Moreover the complete list of quadratic imaginary field with a given class number $h$ is known up to $h=100$ by \cite{Wat}. Hence the table \ref{g2} is unconditionally complete.
  
A key property used in the proof of Proposition \ref{prop:g2} is that a curve of genus $2$ can be defined over a quadratic extension of its field of moduli. The same is true for curves of genus $3$.
\end{rem}

\begin{prop}\label{prop: field of def}
Let $C$ be a curve of genus $3$ with field of moduli $\QQ$. There exist infinitely many quadratic extensions $k/\QQ$ such that $C$ admits a model over $k$.  
\end{prop}
\begin{proof}
  The curve $C$ is either a hyperelliptic curve or a smooth plane quartic. If $C$ is hyperelliptic, it is well known that it can be defined over infinitely many fields $k$ with $[k\colon \QQ]\leq 2$ (the curve $C/\Aut(C)$ is a conic that can be defined over $\QQ$ and $k$ can be taken to be any field where it has rational points). Suppose that $C$ is a plane quartic. If $|\Aut(C)|=1$, the curve admits a model over its field of moduli by Weil descent. If $|\Aut(C)|>2$, then the curve also admits a model over its field of moduli (cf. \cite[Prop. 2.3 and \S3.2]{LRRS}). If $|\Aut(C) |= 2$ the result is Lemma \ref{lemma: field of def aut Z2} below, which is probably well known but we include a proof for completeness.
\end{proof}

\begin{lem}\label{lemma: field of def aut Z2}
  Let $C/\Qbar$ be a curve with field of moduli $\QQ$ and such that $|\Aut(C)|=2$. There exist infinitely many quadratic extensions $k/\QQ$ such that $C$ admits a model over $k$.
\end{lem}
\begin{proof}
  For every $\sigma \in G_\QQ:=\Gal(\Qbar/\QQ)$, let $\mu_\sigma \colon {}^\sigma C \stackrel{\simeq}{\lra} C$ be an isomorphism, chosen in such a way that the system $\{\mu_\sigma\}_{\sigma\in G_\QQ }$ is locally constant. For any $\psi \in \Aut(C)$ we have that
  \begin{align}
    \label{eq:compatible}
    {}^\sigma\psi = \mu_\sigma^{-1} \circ \psi \circ \mu_\sigma.
  \end{align}
  The equality is obvious when $\psi$ is the identity. When $\psi$ is the non-trivial automorphism of $C$, the equality follows from the fact that then ${}^\sigma \psi$ is the non-trivial automorphism of ${}^\sigma C$.

For $\sigma,\tau \in G_\QQ $ define
  \begin{align*}
    c(\sigma,\tau) = \mu_\sigma\circ {}^\sigma\mu_\tau\circ \mu_{\sigma\tau}^{-1}\in \Aut(C).
  \end{align*}
  A short computation using \eqref{eq:compatible} shows that $c$ is a two-cocycle in $Z^2(G_\QQ,\Aut(C))$, where the action of $G_\QQ$ on $\Aut(C)$ is the trivial action. Therefore, the class of $c$ can be identified with an element in $H^2(G_\QQ,\{\pm 1\})$. Now $H^2(G_\QQ,\{\pm 1\})$ is isomorphic to the $2$-torsion of the Brauer group of $\QQ$, which is generated by quaternion algebras. Any product of quaternion algebras is equivalent in the Brauer group to a quaternion algebra, so the class of $c$ can be identified with a quaternion algebra $B$. Let $k$ be any quadratic splitting field of $B$ (there exist infinitely many such fields $k$). The restriction $\Res_\QQ^k(c)$ is a coboundary, that is to say, there exists a function $\delta\colon G_k\ra \{\pm 1\}$ such that
\begin{align*}
  c(\sigma,\tau)=\delta(\sigma)\delta(\tau)\delta(\sigma\tau)^{-1} \text{ for all }\sigma,\tau\in G_k.
\end{align*}
If we define $\tilde\mu_\sigma:=\delta(\sigma)^{-1}\circ\mu_\sigma$ for $\sigma\in G_k$, the system $\{\tilde\mu_\sigma\}_{\sigma\in G_k}$ is a descent data for $C$ over $k$ and we see that $C$ can be defined over $k$.
\end{proof}

\begin{prop}\label{prop:g3}
  Let $E_1,E_2,E_3$ be elliptic curves over $\Qbar$ with CM by $K$. Let $\varphi$ be a principal indecomposable polaritation on $A:=E_1\times E_2\times E_3$ such that $(A,\varphi)$ has field of moduli $\QQ$. Then the class number of $K$ is 1 or 3.
\end{prop}
\begin{proof}
  By the results of  \cite{OU} we have that $(A,\varphi)$ is isomorphic to the canonically polarized Jacobian of some genus three curve $C/\Qbar$. Since $(A,\varphi)$ has field of moduli $\QQ$, the curve $C$ has field of moduli $\QQ$ as well. By Proposition \ref{prop: field of def} there exists a quadratic extension $k'/\QQ$ with $k'\cap {H_K} =\QQ$ such that $C$ admits a model over $k'$. Put $k=k'K$ and suppose from now on that $C$ is defined over $k$.

  If $A= \Jac(C)$ we have that $A$ is defined over $k$ and  $A_{\Qbar} \sim E^3$, where $E$ is an elliptic curve with CM by $K$.  Denote by $L$ the minimal extension of $k$ such that $\End(A_{\Qbar})=\End(A_L)$. By \cite[Theorem 2.14]{FG} the field $F=Hk$ is a subfield of $L$ and $\Gal(F/k)$ has exponent dividing $3$. Since $k\cap {H_K} = K$ we have that $\mathrm{Cl}(K)\simeq \Gal({H_K}/K)\simeq \Gal(F/k)$ has exponent 3.

  The argument in the proof of \cite[Corollary 2.17]{FG} shows that if $v_3(| \Gal(L/k)|) >  1$ then $K = \QQ(\sqrt{-3})$  (here $v_3$ stands for the valuation at $3$). Suppose then that $v_3(| \Gal(L/k)|)\leq 1$. Since $Hk/k$ is a subextension of $L/k$ we see that  $v_3(|\Gal({H_K}/K)|)=v_3(|\Gal(Hk/k)|)\leq 1$. This completes the proof of the proposition.

\end{proof}

\begin{rem}
Let $A$ be an abelian threefold defined over a number field $k$ such that $A_{\Qbar}\sim E^3$, where $E/\Qbar$ is an elliptic curve with CM by $K$. Denote by $L$ the minimal extension of $k$ such that $\End(A_{\Qbar})=\End(A_L)$. As explained in \cite[\S3.2]{FKS21} the group of components $\pi_0(\ST(A))$ of the Sato-Tate group of $A$ can be identified with $\Gal(L/k)$. The last paragraph of the proof of the previous proposition shows that if $v_3(|\pi_0(\ST(A))|)>1$, then $K=\QQ(\sqrt{-3})$. 

There are 4 maximal genus 3 Sato-Tate groups $G$ with connected component of the identity $G^0\simeq \mathrm{U}(1)_3$ and $v_3(|\pi_0(G)|)>1$. These are $J_s(B(T,3))$, $J(B(T,3))$, $J(D(6,6))$, and $J(E(216))$. Abelian threefolds realizing them are given in \cite[\S8]{FKS21}. Consistently with the observation of the previous paragraph, note that in all of these constructions $K$ is taken to be $\QQ(\sqrt{-3})$.
\end{rem}

\bibliographystyle{alpha}
%\bibliography{additionalfiles/biblio.bib}
\newcommand{\etalchar}[1]{$^{#1}$}

\end{document}